\documentclass[acmsmall]{acmart}

\setcopyright{none}
\renewcommand\footnotetextcopyrightpermission[1]{} 
\usepackage[utf8]{inputenc}
\usepackage{graphicx}
\usepackage{amsmath}
\usepackage{amsthm}
\usepackage{amsfonts}
\usepackage{array}
\usepackage{float}
\usepackage{caption}
\usepackage{esvect}
\usepackage{enumitem}
\usepackage{hyperref} 
\usepackage{epigraph} 
\usepackage{url}
\usepackage[all]{xy}
\usepackage{tikz-cd}
\usepackage{faktor}
\usepackage[dutch,english]{babel}
\usepackage{hhline}
\usepackage{relsize}
\usepackage{subcaption} 

\newcommand{\e}{{\bf 1}}

\newcommand*\colvec[3][]{
    \begin{pmatrix}\ifx\relax#1\relax\else#1\\\fi#2\\#3\end{pmatrix}
}

\begin{document}

\title{On the stochastic and asymptotic improvement of First-Come First-Served and
Nudge scheduling}
\author{Benny Van Houdt}

\affiliation{%
  \institution{University of Antwerp}
  \streetaddress{Middelheimlaan 1}
  \city{Antwerp}
  \postcode{2000}
  \country{Belgium}}

\begin{abstract}
Recently it was shown that, contrary to expectations, the First-Come-First-Served (FCFS) scheduling algorithm can be stochastically improved upon by a scheduling algorithm called {\it Nudge}
for light-tailed job size distributions.
Nudge partitions jobs into 4 types based on their size, say small, medium,
large and huge jobs. Nudge operates identical to FCFS, except that whenever a {\it small} job
arrives that finds a {\it large} job waiting at the back of the queue, Nudge swaps the small
job with the large one unless the large job was already involved in 
an earlier swap.

In this paper, we show that FCFS can be stochastically improved upon under far weaker conditions. We consider a system with $2$ job types and limited swapping between type-$1$ and type-$2$ jobs,  but where a type-$1$ job is not necessarily smaller than a type-$2$ job. More specifically,
we introduce and study the Nudge-$K$ scheduling algorithm which allows type-$1$ jobs
to be swapped with up to $K$ type-$2$ jobs waiting at the back of the queue, while type-$2$
jobs can be involved in at most one swap. 
We present an explicit expression for the response time distribution under Nudge-$K$ when
both job types follow a phase-type distribution. Regarding the
 asymptotic  tail improvement ratio (ATIR) , we derive a simple expression for the ATIR,
as well as for the $K$ that maximizes the ATIR. We show that the ATIR is positive and the optimal $K$ tends to infinity in heavy traffic  as long as the type-$2$ jobs are on average longer than the type-$1$ jobs.
\end{abstract}

\fancyfoot{}
\maketitle
\thispagestyle{empty}

\section{Introduction}

Although there exists an abundance of scheduling algorithms, many systems still rely on the 
First-Come-First-Served (FCFS) scheduling algorithm as FCFS is considered to be a fair 
scheduling algorithm that is easy to implement and does not require any job size information. There is also 
theoretical support for selecting FCFS apart from the well-known fact that it minimizes the
maximum response time of any finite sequence of jobs. 
If we denote $R$ as the response time
of an arbitrary job under FCFS and make the following technical assumptions:
\begin{itemize}
\item Jobs arrive according to a Poisson
process.
\item The job size distribution $X$ is light-tailed 
(which means there exists an $\epsilon > 0$ such that
$E[e^{-\epsilon X}]$ is finite).
\item If $\tilde S(s)$ denotes the Laplace transform of the job size distribution,
then $\tilde S(s)$ has either no singularities or if $s^* < 0$ is its right-most singularity, then $\tilde S(s^*)=\infty$. 
\end{itemize}
Then there exist constants $\theta_Z > 0$ and $c_{FCFS} > 0$ such that
\[ P[R > t] \sim c_{FCFS}e^{-\theta_Z t},\]
where $\sim$ indicates that the ratio of the two quantities converges to 1 as $t$ tends to infinity. Note that the latter two assumptions correspond to a class-I distribution
in \cite[Section 5]{abate94} and these distributions include all well-behaved light-tailed distributions such as
any phase-type distribution or distribution with finite support (such as
truncated heavy-tailed distributions).
The constant $\theta_Z$ is called the decay rate. Let $\pi$ be any
scheduling algorithm and $R_{\pi}$ be its associated response time distribution
(in an M/G/1 queue with a class-I job size distribution),
then (see \cite[Section 3.1]{boxma07}) there exists a constant $M(\pi) \geq 0$ such that
\[ \limsup_{t \rightarrow \infty} \frac{P[R>t]}{P[R_\pi > t]} \leq M(\pi).\]
This is equivalent to stating that FCFS has the largest possible decay rate  in systems subject
to Poisson arrivals and class-I job sizes (in fact, $\theta_Z$ is equal to the decay rate of the workload distribution $Z$ in the system). Any scheduling algorithm with the largest possible decay rate is called {\bf weakly tail optimal}. In \cite{wierman12} FCFS was conjectured to be 
{\bf strongly tail optimal} for class-I job sizes, which would imply that $M(\pi) \leq 1$ for any $\pi$ and FCFS results in the best possible tail behavior for 
class-I job sizes. 

In a recent paper \cite{nudge} FCFS was shown {\it not} to be strongly tail optimal by introducing
a scheduling algorithm called Nudge such that $M(Nudge)>1$.  Further, contrary
to expectations, it was shown that the Nudge scheduling algorithm can
stochastically improve upon FCFS. A scheduling algorithm $\pi_1$ is said to
{\bf stochastically improve} upon an algorithm $\pi_2$ if and only if
$P[R_{\pi_1} > t] < P[R_{\pi_2} > t],$
for any $t > x_{min}$, where $x_{min}$ is the infimum of the support of the
job size distribution $X$ (for $t \leq x_{min}$, we have $P[R_{\pi_1} > t] = P[R_{\pi_2} > t]$).
This can be restated by saying that the tail improvement ratio (TIR) in $t$
defined as
\[\mbox{TIR}(t) = 1 - \frac{P[R_{\pi_1} > t]}{P[R_{\pi_2} > t]},\]
is positive for all $t > x_{min}$. This means that Nudge improves every moment and percentile of
the response time  of FCFS! Note that it is easy to devise scheduling algorithms
that reduce the mean response time of FCFS, but this typically comes at the 
expense of a worse decay rate \cite{nuyens08}.

To achieve this stochastic improvement Nudge partitions the jobs into
$4$ types, say small, medium, large and huge jobs,  based on their size
using three thresholds $x_1, x_2$ and $x_3$. Nudge then operates in the same
manner as FCFS, except that whenever a {\it small} job
arrives that finds a {\it large} job waiting at the back of the queue, Nudge swaps the small
job with the large one unless the large job was already involved in 
an earlier swap. The authors of \cite{nudge} then showed that it is possible
for any continuous class-I job size distribution $X$, to find appropriate thresholds $x_1, x_2$
and $x_3$ (that depend on $X$) such that Nudge stochastically improves upon FCFS.
Simulation experiments further showed that this is often still the case if
$x_1=x_2$ and $x_3=\infty$, which means in the absence of medium and huge jobs
(which were needed for the proofs).

In this paper we consider  a system
with two types of jobs (see Section \ref{sec:system} for details), where a random type-$1$ job is not necessarily smaller than a random type-$2$ job, and  a set of scheduling algorithms  called Nudge-$K$, where $K \geq 0$ is an input parameter (that can be set equal to $\infty$).
Under Nudge-$K$ any arriving type-$1$ job can be swapped with at most $K$ type-$2$ jobs waiting at the back of the queue, but type-$2$ jobs can be involved in at most one swap. This means
that a type-$1$ job passes up to $K$ jobs waiting at the back of the queue until it
either encounters another type-$1$ job, a type-$2$ job that was already swapped or becomes
the job waiting at the head of the queue.  Note that Nudge-$1$ coincides with Nudge if
we set $x_1=x_2$, $x_3=\infty$ and call the small jobs type-$1$ and the large jobs type-$2$.

The main contributions of the paper can be summarized as follows:
\begin{enumerate}
\item For the system described in Section \ref{sec:system} we derive an {\it explicit} 
expression for the  complementary cumulative distribution function (ccdf) of the waiting time 
(see Theorems \ref{th:W2} and \ref{th:W1a}) and response time distribution (see Theorems \ref{th:R2} and \ref{th:R1a}) of type-$1$ and type-$2$ jobs. To derive these results we
first obtain an explicit expression for the workload in the system (see Theorem \ref{th:Z}). 
\item We derive a simple expression for  the asymptotic tail improvement ratio of Nudge-$K$ over FCFS defined as
\[\mbox{ATIR}(K) = 1 - \lim_{t \rightarrow \infty} \frac{P[R_{Nudge-K} > t]}{P[R > t]},\]
as well as for the value of $K$ that maximizes the ATIR($K$), denoted as $K_{opt}$ (see Theorem \ref{th:ATIR}). 
We identify simple conditions on when the ATIR($K$) > 0, for a given $K$, for all $K$ and
for $K=1$ (see Theorem \ref{th:ATIRcond}). We further prove
that under heavy traffic, the ATIR$(K) > 0$  provided that the mean
type-$2$ job size exceeds the mean type-$1$ job size and that $K_{opt}$ tends to
infinity (see Theorems \ref{th:KoptHT} and \ref{th:Kapprox}).  
\item We present various novel insights on the stochastic and asymptotic improvement upon
FCFS and Nudge in Section \ref{sec:num} using numerical experiments. These show that
stochastic improvements of FCFS exist under far weaker conditions that the one considered
in \cite{nudge}, that Nudge can be stochastically improved upon, 
that setting $K$ too large may imply that Nudge-$K$ no longer stochastically improves
upon FCFS, that an asymptotic
improvement does not necessarily imply a stochastic improvement even if type-$2$
jobs stochastically dominate type-$1$ jobs, etc.
\end{enumerate}
The fact that FCFS can be stochastically improved upon under far weaker conditions is
important as it is much easier in a real system to identify different types of jobs
such that one job type is typically larger and/or has a heavier tail than another type.
In such case implementing an algorithm like Nudge-$K$ using these types may improve all percentiles of
the response time, without the need of having any indication on the size of individual jobs
(being larger or smaller than some threshold) as in \cite{nudge}.

The paper is structured as follows. The exact model considered in the paper
is presented in Section \ref{sec:system}. A matrix exponential expression
for the workload distribution is derived in Section \ref{sec:Z}, while the
mean response time of Nudge-$K$ is analyzed in Section \ref{sec:mean}.
Explicit expressions for the type-$2$ and type-$1$ waiting and response time distributions
are part of Sections \ref{sec:t2} and \ref{sec:t1}, respectively.
Section \ref{sec:ATIR} contains the results for the ATIR, these results are
 the most elegant results in the paper. Numerical examples and insights are discussed in Section \ref{sec:num}.
Conclusions are drawn and possible future work is listed in Section \ref{sec:concl}.

\section{The system}\label{sec:system}
We consider a queueing system with Poisson arrivals with rate $\lambda$. 
Arriving jobs are type-1 jobs with probability $p$, or type-2 jobs
with probability $1-p$.
Job types of consecutive jobs are independent. The processing time $X_i$ of 
a type-$i$ job  follows  an order $n_i$
phase-type distribution  characterized by $(\alpha_i,S_i)$,
that is, $P[X_i > t] = \alpha_i e^{S_i t} \e$,
where $\e$ is a column vector of ones of the appropriate dimension.
Let $E[X_i] = \alpha_i (-S_i)^{-1} \e$ be the mean service time of a type-$i$ job.
We assume without loss of generality that $E[X]=pE[X_1] + (1-p)E[X_2] = 1$, 
with $X=pX_1+(1-p)X_2$ the job size distribution, such that
the load of the system is $\lambda$. For further use, define $s_i^* = (-S_i)\e$,
$\alpha = (p \alpha_1, (1-p) \alpha_2)$ and
\[S = \begin{bmatrix}
S_1 & 0 \\ 0 & S_2
\end{bmatrix},\]
such that $X$ has a phase-type distribution characterized by $(\alpha,S)$.
Note that $\alpha (-S)^{-1}\e = 1$ as $E[X] = 1$.
Denote $\tilde S_i(s) = \alpha_i (sI-S_i)^{-1} (-S_i)\e$, for $i=1,2$, as the
Laplace transform of the size of a type-$i$ job. Let $\tilde S(s)=
p\tilde S_1(s)+(1-p)\tilde S_{2}(s)$ be the Laplace transform of a random job.
As $X$ is a phase-type distribution, it is a class-I distribution.
It is  well known that any general positive-valued distribution can be approximated arbitrary close with a PH distribution  \cite{latouche1}. Further, various fitting algorithms
and tools are available online 
for phase-type distributions (e.g., \cite{feldman98,panchenko1,Kriege2014}).

The scheduling algorithm studied in this paper is called the {\bf Nudge-$K$ algorithm}.
Under this algorithm a type-$1$ job can be swapped with at most $K$ type-$2$ jobs waiting at the back of the queue and any type-$2$ job can be swapped at most once. In other
words, when a type-$1$ job arrives it can pass up to $K$ waiting type-$2$ jobs
at the back of the queue until it either encounters a type-1 job, a 
type-2 job that was already passed by another type-1 job or becomes the job waiting at
the head of the queue. The job that is being served
is never swapped.


\section{Workload and FCFS response time distribution}\label{sec:Z}

We first provide an
{\bf explicit matrix exponential} expression for the workload distribution, which corresponds to the
waiting time distribution in case of FCFS. The workload distribution
does not depend on the scheduling algorithm, as long as it is work-conserving.
Note that if $Y$ is a class-I distribution, the probability $P[Y > t]$ decays exponentially fast and the decay rate $\theta_Y$
can be expressed as $\theta_Y = -\lim_{t \rightarrow \infty} \frac{1}{t}\log P[Y > t]$. 

\begin{theorem}\label{th:Z}
Let $Z$ be the workload in the system, then
\begin{align}\label{eq:Zt}
P[ Z > t] = \lambda \alpha e^{T t}(-S)^{-1}\e = \lambda \beta e^{T t}(-T)^{-1}\e,
\end{align}
with $\beta = (1-\lambda)\alpha$ and
\begin{align}\label{eq:T}
T = S + \lambda \e \alpha.
\end{align}
Let $\theta_i = -\lim_{t\rightarrow \infty} \frac{1}{t}\log P[X_i > t]$
and $\theta_Z = -\lim_{t\rightarrow \infty} \frac{1}{t}\log P[Z > t]$,
then $0 < \theta_Z < \min (\theta_1,\theta_2)$.
\end{theorem}

\begin{proof}
It is well known (see (5.41) on p247 in \cite{gross1}) that the Pollaczek-Khinchin formula for the Laplace transform
of the workload $\tilde Z(s)$ in an M/G/1 queue can be rewritten as
\[\tilde Z(s) = (1-\lambda) \sum_{n=0}^\infty \lambda^n [\tilde S_{Res}(s)]^n,\]
where $\tilde S_{Res}(s)$ is the Laplace transform of the residual service time.
This implies that
\[P[Z > t] = \lambda \sum_{n=1}^\infty \lambda^{n-1}(1-\lambda) P[S_{Res}^{(n*)} > t],\]
where $S_{Res}^{(n*)}$ is the $n$-fold convolution of the residual service time.
The residual service time of a phase-type distribution 
with representation $(\alpha,S)$ and mean $1$, is also phase-type with representation
$(\alpha (-S)^{-1}, S)$. The probability $P[Z > t]$ can therefore be expressed
as $\lambda$ times a geometric sum of the phase-type distribution $(\alpha (-S)^{-1}, S)$, which is again
phase-type with representation $(\alpha (-S)^{-1}, S+\lambda s^* \alpha (-S)^{-1})$. Hence,
\[P[Z>t] = \lambda \alpha (-S)^{-1} e^{(S+\lambda s^*\alpha (-S)^{-1})t}\e.\] 
We now note that
\begin{align*}
(-S)^{-1} &e^{(S+\lambda s^*\alpha (-S)^{-1})t} = (-S)^{-1} 
\sum_{k=0}^\infty (S+\lambda s^*\alpha (-S)^{-1})^k  t^k \\
&= \sum_{k=0}^\infty \left[ (-S)^{-1} (S+\lambda s^*\alpha (-S)^{-1}) (-S)\right]^k t^k (-S)^{-1}\\
&=\sum_{k=0}^\infty (S+\lambda \e \alpha)^k t^k (-S)^{-1} = e^{(S+\lambda \e \alpha)t}(-S)^{-1},
\end{align*} 
as $s^*=-S\e$.
This shows that
\[P[Z>t] = \lambda \alpha e^{(S+\lambda \e\alpha)t} (-S)^{-1}\e.\] 
The second equality in \eqref{eq:Zt} now follows by noting that we can
use the Sherman-Morrison formula to find that
\begin{align}\label{eq:invrel}
(-T)^{-1}\e = (-S)^{-1}\e + \frac{\lambda (-S)^{-1}\e\alpha(-S)^{-1}\e}
{1-\lambda \alpha (-S)^{-1}\e} = \frac{(-S)^{-1}\e}{1-\lambda},
\end{align} 
as $\alpha(-S)^{-1}\e=E[X]=1$.
The fact that $\theta_Z < \min (\theta_1,\theta_2)$ follows by noting
that for $\zeta > \max_i |S_{ii}|$ the matrix $T+\zeta I$ is a primitive
non-negative matrix with Perron-Frobenius eigenvalue $\zeta-\theta_Z$ \cite{seneta1}.
Therefore for any eigenvalue $\beta$ of a matrix $0 \leq B \leq  T+\zeta I$ with
inequality in at least one entry, 
we have $|\beta| < \zeta-\theta_Z$ by \cite[Theorem 1.1(e)]{seneta1}.
Setting $B=S+\zeta I$ therefore implies that the real eigenvalue
$\zeta - \min(\theta_1,\theta_2)$ of $B$
is strictly smaller than $\zeta-\theta_Z$.
\end{proof}

\paragraph{Remark:} $-\theta_Z <0$ is a real eigenvalue of $T$ and for any other
eigenvalue $\xi$ of $T$ we have $Re(\xi) < -\theta_Z$. $-\theta_Z$ may not be
equal to the spectral radius of $T$ (that is, $|\xi| \leq -\theta_Z$ may not hold), 
but $-\theta_Z t$ is the spectral radius of $e^{Tt}$.
Further, as $T$
is irreducible, 
it has a unique right and left eigenvector (up to multiplication by a constant)
associated with $-\theta_Z$ and these eigenvectors can be chosen strictly positive
\cite{seneta1}. 
If we denote these unique eigenvectors as $u_T$ and $v_T^*$ and normalize such that $v_T^* u_T=1$, then
\[\lim_{t \rightarrow \infty} e^{\theta_Z t} e^{Tt} = u_Tv_T^*,\]
when $T$ is irreducible, meaning
\[c_Z = \lim_{t \rightarrow \infty} e^{\theta_Z t} P[Z > t] = \lambda (\beta u_T) (v_T^* (-T)^{-1}\e),\]

To prove the next theorem we rely on the following Lemma:

\begin{lemma}[Theorem 1 in \cite{cvl1}]\label{lem:cvl}
Let 
\[A = \begin{bmatrix}
A_{11} & A_{12} \\ 0 & A_{22}
\end{bmatrix},\]
then
\[e^{At} = \begin{bmatrix}
e^{A_{11}t} & 
\int_0^t e^{A_{11}s} A_{12} e^{A_{22}(t-s)} ds \\ 0 & e^{A_{22}t}
\end{bmatrix}.\]
\end{lemma}

\begin{theorem}\label{th:FCFS}
Let $R$ be the response time distribution of a job in case of FCFS, then
\begin{align}
P[ R > t] =  (1-\lambda) \alpha e^{S t} \e  + \lambda
(\beta, 0) e^{Ut} \colvec{(-T)^{-1}\e}{\e},
\end{align}
with
\[U= \begin{bmatrix}
T & \e \alpha \\ 0 & S
\end{bmatrix}.\]
Further, $\theta_Z = -\lim_{t\rightarrow \infty} \frac{1}{t}\log P[R > t]$
and $c_{FCFS} =\lim_{t \rightarrow \infty} e^{\theta_Z t}P[R > t]$ can be expressed as
\[ c_{FCFS} = c_Z \tilde S(-\theta_Z).\]
\end{theorem}
\begin{proof}
The response time $R$ in case of FCFS is given by the
workload $Z$ plus the job size $X$. The density of the workload $Z$ is given by
$\lambda \beta e^{Ts}\e$ due to \eqref{eq:Zt}, hence
\begin{align*}
P[ R > t] &= P[Z > t] + 
\lambda \int_{0}^t \beta e^{Ts} \e 
\alpha e^{S(t-s)} \e ds + (1-\lambda) \alpha e^{St} \e,
\end{align*}
as a job finds a workload equal to zero with probability $1-\lambda$.
The result for $P[R>t]$ now follows using Lemma \ref{lem:cvl}, \eqref{eq:Zt} 
and by combining both matrix exponentials.

As the response time $R$ equals the workload $Z$ plus the service time $X$ which
is independent of the workload, the Laplace transform 
of the response time $\tilde R(s) =
\tilde Z(s) \tilde S(s)$. 
Applying the final value theorem to $e^{\theta_Z t}P[R>t]$ we therefore have
\begin{align*}
c_{FCFS}&=\lim_{t\rightarrow \infty} e^{\theta_Z t}P[R>t]=\frac{1}{\theta_Z}
 \lim_{s \rightarrow 0} s \tilde R(s-\theta_Z) = \frac{1}{\theta_Z}
 \lim_{s \rightarrow 0} s \tilde Z(s-\theta_Z) \tilde S(s-\theta_Z) \\
 &=\lim_{t\rightarrow \infty} e^{\theta_Z t}P[Z>t] \tilde S(-\theta_Z) =
 c_Z \tilde S(-\theta_Z),
\end{align*}
where we used the time-domain integration and frequency shifting properties of
the Laplace transform (in the 2nd and 4th equality).
\end{proof}

We can also argue that
\[c_{FCFS} = \lim_{t \rightarrow \infty} e^{\theta_Z t} P[R > t] = 
\lambda (\beta, 0) u_U v_U^* \colvec{(-T)^{-1}\e}{\e},\]
where $u_U$ and $v^*_U$ are the unique right and left eigenvectors
of $U$ associated with the eigenvalue $-\theta_Z$ such that
$v_U^*u_U = 1$.

\section{Mean Response time of Nudge-$K$}\label{sec:mean}

Let $e_i$ and $e_i^*$ represent the $i$-th column and $i$-th row of 
the size $K+1$ identity matrix, respectively.
\begin{lemma}\label{lem:pswap}
Given that a type-$2$ job sees a workload of $s > 0$ upon arrival, it
is swapped with probability
\begin{align}\label{eq:pswaps}
p_{swap}(s) = e_1^* e^{Ms} e_{K+1},
\end{align}
where $M$ is a size $K+1$ matrix with entry $m_{ij}$ given by
\[m_{ij} = \left\{ \begin{array}{ll}
-\lambda & 1\leq i=j \leq K, \\
\lambda (1-p) & 1 \leq i < K, j=i+1,\\
\lambda p & 1 \leq i \leq K, j=K+1, \\
0 & \mbox{otherwise}
\end{array}\right..
\]
For $K=\infty$, we have $p_{swap}(s)=1-e^{-\lambda p s}$.
\end{lemma}
\begin{proof}
Consider the continuous time Markov chain with upper-triangular rate matrix
\begin{align}
Q = \begin{bmatrix}
M & \lambda(1-p)e_{K+1}\\
0 & 0
\end{bmatrix} = \left[\begin{array}{@{}cccc|cc@{}}
-\lambda & \lambda(1-p) & & & \lambda p & \\
 & \ddots & \ddots & &\vdots   & \\
 & & -\lambda &  \lambda (1-p) & \lambda p & \\
  & & & -\lambda  & \lambda p & \lambda (1-p) \\ \hline
  & &  &  & 0 &  \\
  & &  &  &  & 0 
\end{array}\right].
\end{align}
This $K+2$-state Markov chain has two absorbing states: state $K+1$ and $K+2$. 
Entry $(i,j)$ of $e^{Qs}$, with $i \leq j \leq K$, represents the probability
there are exactly $j-i$ type-$2$ arrivals and zero type-$1$ arrivals in
an interval of length $s$. Entry $(i,K+2)$ holds the probability that
more than 
$K-i$ arrivals occur in an interval of length $s$ and the first
$K-i+1$ arrivals are type-$2$, while entry $(i,K+1)$ contains the remaining
probability mass and corresponds to the probability that there is a type-$1$
arrival in an interval of length $s$ that is preceded by at most $K-i$ type-$2$ arrivals.

A type-$2$ job  is swapped if there is
a type-$1$ arrival that is preceded by at most $K-1$ type-$2$ arrivals while
it is waiting. Hence, $p_{swap}(s)$ can be expressed as entry $(1,K+1)$ of
$e^{Qs}$, which is identical to entry $(1,K+1)$ of $e^{Ms}$ and can be written
in matrix form as $e_1^* e^{Ms} e_{K+1}$. 

The expression for $K=\infty$ is immediate by noting that a type-$2$ job is swapped as soon
as a single type-$1$ arrival occurs during its waiting time, irrespective of whether and
where type-$2$ arrivals occur.
\end{proof}

\begin{theorem}\label{cor:ER}
The mean response time $E[R_{\mbox{Nudge-}K}]$ of the Nudge-$K$ algorithm can be expressed as
\begin{align*}
E[R_{\mbox{Nudge-}K}] = E[R] + (1-p) p_{swap}
(E[X_1]-E[X_2]),
\end{align*}
with
\begin{align*}
p_{swap} = -\lambda(\beta \otimes e_1^*) (T\oplus M)^{-1} (\e \otimes e_{K+1}),
\end{align*}
and $E[R] = 1+\lambda \beta T^{-2} \e = 1+\frac{\lambda E[S^2]}{2(1-\lambda)}$.  
When $K=\infty$ we have $p_{swap}=\lambda(1-\beta (\lambda p I-T)^{-1}\e)$.
\end{theorem}
\begin{proof}
A swap between a type-$1$ and a type-$2$ job changes the mean response time
by $E[X_1]-E[X_2]$ on average as the response time of the type-$2$ job increases by
$E[X_1]$ on average and that of the type-$1$ job decreases by $E[X_2]$ on average.
The rate at which swaps occur equals $\lambda$ times $(1-p)$ times the probability 
$p_{swap}$ that a random type-$2$ job is swapped. Making use of Theorem
\ref{th:Z} and  Lemma \ref{lem:pswap},
we have
\begin{align*}
p_{swap} &= \lambda \int_0^\infty \beta e^{Ts} \e p_{swap}(s) ds  = \lambda \int_0^\infty (\beta \otimes e_1^*) (e^{Ts} \otimes e^{Ms}) (\e \otimes e_{K+1}) ds\\
&= \lambda \int_0^\infty (\beta \otimes e_1^*) e^{(T \oplus M)s} (\e \otimes e_{K+1}) ds= -\lambda (\beta \otimes e_1^*) (T\oplus M)^{-1} (\e \otimes e_{K+1}).  
\end{align*} 
Finally, the mean response time is found as the mean response time $E[R]$ in case 
of FCFS plus the mean number of swaps that occur per arrival times the
average change $(E[X_1]-E[X_2])$ caused by a swap.    
The mean number of swaps that occur per arrival is clearly given by the
ratio between the swap rate $\lambda (1-p)p_{swap}$ and the arrival rate 
$\lambda$. 

The expression for $E[R]$ in terms of $E[S^2]$ is well known. The other expression 
follows from Theorem \ref{th:Z} and can also be obtained directly using \eqref{eq:invrel} and
the fact that $E[S^2]=2 \alpha (-S)^{-2}\e$.
\end{proof}

\paragraph{Remark:} The mean response time of Nudge-$K$ is smaller than
the mean response time $E[R]$ under FCFS if and only if $E[X_2]> E[X_1]$.
This implies that a stochastic improvement is only possible if
type-$2$ jobs are on average larger than type-$1$ jobs.

The probability $p_{swap}(s)$ is clearly increasing in $K$ and therefore
so is $p_{swap}$. This implies that setting $K=\infty$ minimizes the
mean response time of Nudge-$K$ provided that $E[X_2]>E[X_1]$. 
We will see however that this choice for
$K$ is {\it never} optimal if we focus on the asymptotic tail improvement ratio.

\section{Response time distribution of type-2 jobs}\label{sec:t2}

We now proceed with the waiting time distribution of type-2 jobs.

\begin{theorem}\label{th:W2}
Let $W^{(2)}$ be the waiting time distribution of a type-2 job, then
\begin{align}\label{eq:W2}
P[ W^{(2)} > t] = P[Z > t] + \lambda (\beta \otimes e_1^*, 0) e^{T^{(2)}t} 
\colvec{0_m}{\e},
\end{align}
with $0_m$ a column vector of zeros with the same size as $(T\oplus M)\e$,
\[T^{(2)} = \begin{bmatrix}
T \oplus M & \e \otimes e_{K+1}\alpha_1 \\ 0 & S_1
\end{bmatrix},\]
where $M$ is defined in Lemma \ref{lem:pswap}.

Further, $-\lim_{t\rightarrow \infty} \frac{1}{t}\log P[W^{(2)} > t] = \theta_Z$
and $c_{W^{(2)}} = \lim_{t \rightarrow \infty} e^{\theta_Z t} P[W^{(2))} > t]$ is
given by
\begin{align}\label{eq:cW2}
c_{W^{(2)}}  = (1-p)^K c_Z + (1-(1-p)^K) c_Z \tilde S_1(-\theta_Z).
\end{align}
\end{theorem}
\begin{proof}
The waiting time of a type-$2$ job equals the workload $Z$ in the queue when it arrives
plus the workload of a type-$1$ job if the type-$2$ job is swapped.
As $\lambda \beta e^{Ts} \e$ is the density of the workload and a type-$2$ job that
sees a workload of $s$ is swapped with probability $p_{swap}(s)=e_1^* e^{Ms} e_{K+1}$
due to Lemma \ref{lem:pswap}, we get
\begin{align*}
P[ W^{(2)} > t] &= P[Z > t] + \lambda 
\int_0^t \beta e^{Ts} \e (e^*_1 e^{Ms} e_{K+1}) \alpha_1 e^{S_1(t-s)}\e ds,\\
&= P[Z > t] + \lambda 
\int_0^t (\beta \otimes e_1^*) (e^{Ts} \otimes e^{Ms}) (\e \otimes e_{K+1}) \alpha_1 e^{S_1(t-s)}\e ds\\
&= P[Z > t] + \lambda 
\int_0^t (\beta \otimes e_1^*) e^{(T \oplus M) s} (\e \otimes e_{K+1}) \alpha_1 e^{S_1(t-s)}\e ds.
\end{align*}
Applying Lemma \ref{lem:cvl} then yields \eqref{eq:W2}.
The first limit for $t$ tending to infinity follows from noting that
the eigenvalue of $e^{T^{(2)}t}$ with the largest real part
is given by the eigenvalue with the largest real part of 
the matrices $e^{(T \oplus M)t}$ and $e^{S_1 t}$. As $e^{(T \oplus M)t} = e^{Tt} \otimes e^{Mt}$,
the eigenvalues of $e^{(T \oplus M)t}$ are the products of the eigenvalues
of $e^{Tt}$ and $e^{Mt}$. The eigenvalues of $e^{Mt}$ are 
$1$ and  $e^{-\lambda t}$ (with multiplicity $K$), while the eigenvalue with the largest
real part of $e^{Tt}$ equals $-\theta_Z t$ by definition. Hence, $-\theta_Z t$
is the eigenvalue with the largest real part of $e^{(T \oplus M)t}$ and
therefore also of $e^{T^{(2)}t}$ as $\theta_Z < \theta_1$. 

The expression for $c_{W^{(2)}}$ follows from first noting that
$p_{swap}(s)$ is increasing in $s$ and therefore
\[ e^*_1 e^{Ms} e_{K+1} \leq \lim_{s \rightarrow \infty} e^*_1 e^{Ms} e_{K+1} = 1-(1-p)^K. \]
Thus for any $\epsilon > 0$ there exist a $t_\epsilon$ such that
\[1-(1-p)^K -\epsilon < e^*_1 e^{Ms} e_{K+1} < 1-(1-p)^K, \]
for $s > t_\epsilon$. Further, 
\[
\lim_{t \rightarrow \infty} e^{\theta_Z t} \int_0^{t_\epsilon} \beta  e^{Ts} \e (e^*_1 e^{Ms} e_{K+1}) \alpha_1 e^{S_1(t-s)}\e ds = 0, \]
as $\alpha e^{S_1 t}\e$ decays faster than $e^{-\theta_Z t}$. 
This implies that 
\begin{align*}
\lim_{t \rightarrow \infty}& e^{\theta_Z t} P[ W^{(2)} > t] =\\ 
&\lim_{t \rightarrow \infty} e^{\theta_Z t} \left( P[Z >t] + (1- (1-p)^K)
\int_0^t (\lambda\beta e^{Ts}\e) \alpha_1 e^{S_1(t-s)}\e ds \right)=\\ 
&\lim_{t \rightarrow \infty} e^{\theta_Z t} \left( (1-p)^K P[Z > t]+
(1-(1-p)^K) P[Z+X_1 > t]\right),
\end{align*} 
from which the expression for $c_{W^{(2)}}$ follows by using the final value theorem
in the same manner as in the proof of Theorem \ref{th:FCFS}.
\end{proof}

Remark that when $K=\infty$ we can still use the above approach by simply
replacing the matrix $M$ by the $2 \times 2$ matrix
\begin{align*}
M = \begin{bmatrix}
-\lambda p & \lambda p\\
0 & 0.
\end{bmatrix},
\end{align*}
and $e_{K+1}$ by $e_2$.
The same remark applies to Theorem \ref{th:R2}.

\begin{theorem}\label{th:R2}
Let $R^{(2)}$ be the response time distribution of a type-2 job, then
\begin{align}
P[ R^{(2)}& > t] = (1-\lambda) \alpha_2 e^{S_2 t} \e 
+ \lambda (\beta,0) e^{U_2^{(2)}t}\colvec{(-T)^{-1}\e}{\e}
+ \lambda
(\beta \otimes e_1^*,0) e^{U_1^{(2)}t}\colvec{0_m}{\e},
\end{align}
with
\[U_1^{(2)} = \begin{bmatrix}
T \oplus M & \e \otimes e_{K+1}\alpha_1  & -\e \otimes e_{K+1}\alpha_2 \\ 0 & S_1 & s_1^* \alpha_2 \\
0 & 0 & S_2
\end{bmatrix},
\]
and
\[U_2^{(2)} = \begin{bmatrix}
T  & \e \alpha_2  \\  0 & S_2
\end{bmatrix}.
\]
Further, $-\lim_{t\rightarrow \infty} \frac{1}{t}\log P[R^{(2)} > t] = \theta_Z$
and $c_{R^{(2)}} = \lim_{t \rightarrow \infty} e^{\theta_Z t} P[R^{(2)} > t]
= c_{W^{(2)}}\tilde S_2(-\theta_Z)$.
\end{theorem}
\begin{proof}
The result follows from noting that
\begin{align}\label{eq:R2}
P[ R^{(2)} > t] &= P[Z > t] + \lambda 
\int_0^t \beta e^{Ts} \e (e^*_1 e^{Ms} e_{K+1}) 
P[X_1+X_2 > t-s] ds,\nonumber\\
&+(1-\lambda) \alpha_2 e^{S_2 t} \e  +\lambda \int_0^t \beta e^{Ts} \e (1-e^*_1 e^{Ms} e_{K+1}) 
\alpha_2 e^{S_2(t-s)}\e ds,
\end{align}
with 
\[P[X_1+X_2 > t-s]=(\alpha_1,0) e^{\begin{bmatrix}
S_1 & s_1^* \alpha_2 \\
0 & S_2 \end{bmatrix}(t-s)}\e.\]
The result now follows 
using Lemma \ref{lem:cvl} and combining some of the matrix exponentials.
An alternate proof exists in making use of the fact that
\begin{align*}
P[R^{(2)} > t] = (1-\lambda) &\alpha_2 e^{S_2 t} \e + P[W^{(2)} > t] + \int_0^t \left( -\frac{\partial}{\partial s} P[W^{(2)} > s] \right)
\alpha_2 e^{S_2 (t-s)} \e ds.
\end{align*}
The equality   $c_{R^{(2)}}=c_{W^{(2)}}\tilde S_2(-\theta_Z)$ follows from the
final value theorem.
\end{proof}

Apart from \eqref{eq:cW2} we can also express $c_{W^{(2)}}$ as
\begin{align*}
c_{W^{(2)}} &= c_Z + \lambda (\beta \otimes e_1^*, 0) u_{T^{(2)}} v^*_{T^{(2)}}
\colvec{0_m}{\e},
\end{align*}
where $u_{T^{(2)}}$ and $v^*_{T^{(2)}}$ are the unique right and left eigenvectors
of $T^{(2)}$ associated with the eigenvalue $-\theta_Z$ such that
$v_{T^{(2)}}^*u_{T^{(2)}} = 1$.


\section{Response time distribution of type-1 jobs}\label{sec:t1}
In this section we present an approach for the waiting time and response time distribution
of a type-1 job.
The approach exists in computing the {\bf workload} distribution seen by a type-1 job from the {\bf queue length} distribution of the FCFS queue. The following observation make this possible.
\begin{enumerate}
\item Whenever a type-2 job is waiting in the FCFS queue, it is also waiting in
the Nudge-$K$ queue.
\item If the FCFS queue contains $i$ type-$2$ jobs waiting at the back of the queue,
these $i$ jobs  are also waiting at the back of the Nudge-$K$ queue (as there have been no type-$1$
arrivals after these type-$2$ arrivals).  
\item If the $i+1$ jobs waiting at the back of the FCFS queue are a type-$1$ job, say job $j$, followed by $i$ type-$2$ jobs, then any new type-$1$ arrival passes exactly $\min(i,K)$ type-$2$
jobs in the Nudge-$K$ queue. Note that job $j$ may have passed one or multiple type-$2$
jobs in the Nudge-$K$ queue, but in such case the $(i+1)$-th last job in the Nudge-$K$ queue
is a type-$2$ job that was already swapped.     
\end{enumerate}   
Hence, we conclude that the number of type-$2$ jobs that a tagged type-$1$
job passes under Nudge-$K$ is equal to the minimum of $K$ and the number of type-$2$ jobs that are
{\bf waiting at the back of the FCFS queue} when the tagged type-$1$ job arrives.
Note that this argument fails if a type-$2$ job can be passed by more than one type-$1$ job,
which is not the case under Nudge-$K$.

\begin{theorem}\label{th:W1a}
Let $W^{(1)}$ be the waiting time distribution of a type-1 job.
Let $R=-\lambda (S-\lambda I +\lambda \e \alpha)^{-1}$ and
$\pi_1 = (1-\lambda) \alpha R$.
Further define
\begin{align} \label{eq:pi01}
\pi_0^{(1)} &= \pi_1 (I-(1-p)^{K+1}R^{K+1})(I-(1-p)R)^{-1},\\
\pi_1^{(1)} &= \pi_1 (1-p)^K R^{K+1} ,\label{eq:pi11} \\
\pi_0^{(2)} &= \pi_1 R (I-(1-p)^K R^K)(I-(1-p)R)^{-1} p,\label{eq:pi02}
\end{align}
then
\begin{align}\label{eq:PW1_th}
P[ W^{(1)} > t] &= \nu_1 \mathlarger{e}^{\mathlarger{(S^\top \otimes I + (s^* \alpha)^\top \otimes R) t}} \xi + (\pi_0^{(1)}-\pi_1^{(1)}R^{-1}) e^{St}\e - \pi_0^{(2)}R^{-1} e^{S t} \colvec{\e/p}{0},
\end{align}
where $X^\top$ denotes the transposed of $X$,
\begin{align*}
\nu_1 &= \e^\top \otimes [(\pi_1^{(1)}+\pi_0^{(2)}) (I-R)^{-1}+\pi_1^{(1)}R^{-1}] + (\e^\top,  0)/p \otimes \pi_0^{(2)} R^{-1}, 
\end{align*}
and $\xi$ is a size $(n_1+n_2)^2$ vector obtained from stacking the columns of the
size $n_1+n_2$ unity matrix.

Further, $-\lim_{t\rightarrow \infty} \frac{1}{t}\log P[W^{(1)} > t] = \theta_Z$
and $c_{W^{(1)}} = \lim_{t \rightarrow \infty} e^{\theta_Z t} P[W^{(1))} > t]$ is
given by
\begin{align}\label{eq:cW1}
c_{W^{(1)}}  = c_Z  (1-p)^{K} \tilde S(-\theta_Z)^{-K} + c_Z
p \frac{\tilde S_1(-\theta_Z)}{\tilde S(-\theta_Z)} \frac{1 - (1-p)^K \tilde S(-\theta_Z)^{-K}}{1 - (1-p) \tilde S(-\theta_Z)^{-1}}.
\end{align}
\end{theorem}
\begin{proof}
Let $P[Q^{(FCFS)}=(q,i)]$ be the steady state probability that the FCFS queue
contains $q$ jobs and the server is in service phase $i \in \{1,2,\ldots,n_1+n_2\}$, for $q > 0$. The probability that
the queue is empty is clearly $1-\lambda$.
It is well known \cite{neuts2} that 
\begin{align*}
P[Q^{(FCFS)}=(q,i)] = (\pi_1 R^{q-1})_i,
\end{align*} 
with $R=-\lambda (S-\lambda I +\lambda \e \alpha)^{-1}$ and
$\pi_1 = (1-\lambda) \alpha R$.
    
As noted earlier, the workload seen by a tagged type-1 job corresponds to the work present at
a random point in time in the FCFS queue if we remove up to $K$ type-2 jobs
that are waiting at the back of the FCFS queue. Note that if fewer than $K$ jobs are removed
because of the presence of another type-1 job, then this
should be taken into account when computing the workload. 

We now define two reduced queue length distributions. The first $Q_1$ 
corresponds to the case where either 
\begin{enumerate}
\item The FCFS queue contains $1 \leq k \leq K+1$ jobs upon arrival of the tagged type-$1$ job
and the $k-1$ waiting jobs
are type-$2$ jobs (that are swapped with the 
tagged type-$1$ job by the Nudge-$K$ algorithm). This occurs when the FCFS server is in phase $i$ 
with probability $(\pi_k)_i (1-p)^{k-1}$.   
\item The FCFS queue contains exactly $q+1+K$ jobs, with $q>0$, upon arrival of the
tagged type-$1$ job and
the last $K$ waiting jobs are type-$2$ jobs (that are swapped with the 
tagged type-$1$ job). This happens when the FCFS server is in phase $i$ with
probability $(\pi_{K+1+q})_i (1-p)^K$.

\end{enumerate}
In both cases the {\bf workload} seen by the tagged type-$1$ job is nonzero and
corresponds to the sum of $q \geq 0$ jobs plus a remaining service time which if the
current 
phase equals $i$, has a phase-type distribution given by $(e_i,S)$.
Hence,
\begin{align}\label{eq:PQ1}
P[Q_1 = (q,i)] = 1[q=0] \sum_{k=1}^{K+1} (\pi_k)_i  (1-p)^{k-1}
+ 1[q > 0] (\pi_{K+1+q})_i (1-p)^K.
\end{align}
The second $Q_2$ corresponds to the case where the tagged job is 
swapped with $0 \leq k < K$ jobs and there is at least one type-$1$
job waiting in the FCFS queue. In this case the workload consists of 
the sum of $q \geq 0$ jobs, one type-1 job and a remaining service time with
phase-type distribution $(e_i,S)$, if the server is in phase $i$:
\begin{align}\label{eq:PQ2}
P[Q_2 = (q,i)] = \sum_{k=0}^{K-1} (\pi_{q+k+2})_i  (1-p)^{k}p.
\end{align}
Note that $Q_1$ and $Q_2$ have a matrix geometric form with the same
rate matrix $R$ as the FCFS queue. More specifically,
let the vectors $\pi_q^{(j)}$ contain the probabilities $P[Q_j = (q,i)]$,
for $j=1,2$, then $\pi_q^{(1)} = \pi_1^{(1)} R^{q-1}$ for $q > 0$ and
   $\pi_q^{(2)} = \pi_0^{(2)} R^{q}$ for $q \geq 0$. 
The expressions in \eqref{eq:pi01}, \eqref{eq:pi11} and \eqref{eq:pi02}
follow from \eqref{eq:PQ1} and \eqref{eq:PQ2}.
   
The probability $P[W^{(1)} > t]$ that the workload seen by a tagged
type-1 job exceeds $t$ can now be computed as
 \begin{align}\label{eq:PW1t}
 P[W^{(1)} > t] &= \sum_{q=0}^\infty 
 P[Q^{(1)}=(q,i)] P[X^{(q*)} + R_i >t] \nonumber \\
 &+\sum_{q=0}^\infty 
P[Q^{(2)}=(q,i)] P[X^{(q*)} + X_1 + R_i >t],
\end{align}    
where $R_i$ is a random variable with phase-type distribution $(e_i,S)$, 
$X^{(q*)}$ is the sum of $q$ independent copies of the job size with
phase-type representation $(\alpha,S)$
 and $X_1$ is the type-$1$ job size with phase-type representation $(\alpha_1,S_1)$.

The probabilities $P[X^{(q*)} + R_i >t]$ and
$P[X^{(q*)} + X_1 + R_i >t]$ can be expressed using the following two observations.
$P[X^{(q*)} + R_i >t]$ is the probability that there are less than
$q+1$ renewals in $[0,t]$ for the phase-type renewal process with
inter-renewal time $(\alpha,S)$ that starts in
phase $i$ at time zero. The probability $P[X^{(q*)} + X_1 + R_i >t]$ can be expressed as
$P[X^{(q*)} + R_i >t] + P[X^{(q*)} + R_i < t, X^{(q*)} + R_i + X_1 > t]$.
Let $P(k,t)$ be the matrix such that entry $(i,j)$ contains the probability
that $k$ renewals occur in $[0,t]$ given that the initial phase equals $i$
and the phase at time $t$ is $j$ of the phase-type renewal process with 
inter-renewal time $(\alpha,S)$. Then,
\begin{align*}
\sum_{q=1}^\infty & P[Q^{(1)}=(q,i)] P[X^{(q*)} + R_i >t] 
= \sum_{q=1}^\infty \pi_1^{(1)} R^{q-1} \sum_{k=0}^q
P(k,t) \e \\
&= \pi_1^{(1)}  \sum_{q=1}^\infty R^{q-1} \sum_{k=0}^{q-1} P(k,t) \e
+ \pi_1^{(1)}R^{-1}  \sum_{q=1}^\infty R^q P(q,t) \e\\
&= \pi_1^{(1)} (I-R)^{-1} \sum_{k=0}^{q-1} R^k P(k,t) \e
+ \pi_1^{(1)}R^{-1}  \left(\sum_{k=0}^\infty R^k P(k,t) \e - P(0,t)\e \right)\\
&=\pi_1^{(1)} ((I-R)^{-1} +R^{-1})  \sum_{k=0}^\infty R^k P(k,t) \e -
\pi_1^{(1)}R^{-1} e^{St} \e. 
\end{align*}
Using the same reasoning as in the proof of \cite[Theorem 2]{ozawa1}, we can show
that 
\begin{align}\label{eq:ozawa}
a \sum_{k=0}^\infty R^k P(k,t) b= (b^\top \otimes a) e^{(S^\top \otimes I + (s^* \alpha)^\top \otimes R) t} \xi,
\end{align}
for any row vector $a$ and column vector $b$ as
\begin{align*}
\frac{\partial}{\partial t} P(0,t) &= SP(0,t), \\
\frac{\partial}{\partial t} P(k,t) &= SP(k,t)+s^*\alpha P(k-1,t),
\end{align*}
with $P(0,0)=I$ and $P(k,0)=0$ for $k > 0$. 
Setting $b=\e$ and $a=\pi_1^{(1)} ((I-R)^{-1} +R^{-1})$ implies
that
\begin{align}\label{eq:PWt1_sum1}
\sum_{q=0}^\infty & P[Q^{(1)}=(q,i)] P[X^{(q*)} + R_i >t] =
(\pi_0^{(1)}-\pi_1^{(1)}R^{-1}) e^{St}\e \nonumber \\
&+ (\e^\top \otimes \pi_1^{(1)} ((I-R)^{-1} +R^{-1})) e^{(S^\top \otimes I + (s^* \alpha)^\top \otimes R) t} \xi.
\end{align}
For the second sum in \eqref{eq:PW1t} it is worth noting that the events
$X^{(q*)} + R_i \leq t$ and $X^{(q*)} + X_1 + R_i >t$ occur simultaneously if there are
exactly $q+1$ renewals for the phase-type renewal process in $[0,t]$ starting from
phase $i$ given that the $(q+2)$-th inter-renewal time corresponds to a type-$1$ job.
In other words $P[X^{(q*)} + R_i \leq t, X^{(q*)} + X_1 + R_i >t] = \sum_{j=1}^{n_1}
P(q+1,t)_{ij}/p$ as the first $n_1$ phases correspond to a type-$1$ job and the
fraction of type-$1$ jobs equals $p$.
With this observation we can 
express the second sum in \eqref{eq:PW1t} as
\begin{align*}
\sum_{q=0}^\infty & P[Q^{(2)}=(q,i)] P[X^{(q*)} + X_1 + R_i >t] \\
&=\sum_{q=0}^\infty  \sum_i (\pi_0^{(2)} R^q)_i \left(P[X^{(q*)} + R_i > t]
 +
P[X^{(q*)} + R_i \leq t, X^{(q*)} + X_1 + R_i >t]\right)\\
&=\sum_{q=0}^\infty   \pi_0^{(2)} R^q\left(\sum_{k=0}^q P(k,t) \e+
P(q+1,t) \colvec{\e/p}{0}\right)\\
&= \pi_0^{(2)} (I-R)^{-1} \sum_{k=0}^\infty R^kP(k,t) \e+
 \pi_0^{(2)} R^{-1} \left(\sum_{k=0}^\infty R^{k} P(k,t) - P(0,t)\right) \colvec{\e/p}{0}.
\end{align*}
Due to \eqref{eq:ozawa} we have
\begin{align}\label{eq:PWt1_sum2}
\sum_{q=0}^\infty & P[Q^{(2)}=(q,i)] P[X^{(q*)} + X_1 + R_i >t] \nonumber\\
&= [(\e^\top \otimes \pi_0^{(2)} (I-R)^{-1}) + ((\e^\top/p, 0)  \otimes \pi_0^{(2)}R^{-1}) ]
e^{(S^\top \otimes I + (s^* \alpha)^\top \otimes R) t} \xi  -\pi_0^{(2)}R^{-1} e^{St}\colvec{\e/p}{0}.
\end{align}
Combining \eqref{eq:PW1t}, \eqref{eq:PWt1_sum1} and \eqref{eq:PWt1_sum2} yields \eqref{eq:PW1_th}.

As the workload decays at rate $\theta_Z$ and $\pi_1 \sum_{k=0}^\infty R^k P(k,t) \e$ is
the probability that the workload exceeds $t$, \eqref{eq:ozawa} implies
that $e^{(S^\top \otimes I + (s^* \alpha)^\top \otimes R) t} $ decays at rate $\theta_Z$,
while the other terms decay faster.
The expression for $c_{W^{(1)}}$ can be derived by noting that for any vector $a$
we have
\[a \sum_{k=0}^q P(k,t) \e = (a,0) e^{\Omega t} \e,\] 
where the block matrix 
\[\Omega = \begin{bmatrix}
S & s^* \alpha & & \\
 & S & \ddots & \\
& & \ddots & s^* \alpha \\ & & & S
\end{bmatrix},\]
has $q+1$ diagonal blocks. We therefore
have that $a \sum_{k=0}^q P(k,t) \e$ for any fixed $q$ has a decay rate equal
to $\min(\theta_1,\theta_2)$. The probability $P[R_i > t]$ also decays at this
rate. We may therefore when computing $\lim_{t \rightarrow \infty} e^{\theta_Z t}
P[W^{(1)}>t]$ start the summations in \eqref{eq:PW1t} at $K$. 
This means that as far as this limit is concerned, 
a tagged type-1 job passes $N$ type-2 jobs, where $N$ is a truncated geometric distribution, that is,
$P[N=k]=p(1-p)^{k}$ for $k <K$ and $P[N=K]=(1-p)^K$. Applying the final value theorem 
in the same manner as before therefore 
yields
\begin{align*} 
c_{W^{(1)}} &= c_Z  (1-p)^K \tilde S_2(-\theta_Z)^{-K} \frac{\tilde S_2(-\theta_Z)^K}{\tilde S(-\theta_Z)^{K}}
 + c_Z
\sum_{k=0}^{K-1} p(1-p)^k \tilde S_2(-\theta_Z)^{-k} \frac{\tilde S_2(-\theta_Z)^k \tilde S_1(-\theta_Z)}{\tilde S(-\theta_Z)^{k+1}}\\
&=  c_Z \left( (1-p)^K \tilde S(-\theta_Z)^{-K} + p \frac{\tilde S_1(-\theta_Z)}{\tilde S(-\theta_Z)}
\frac{1-(1-p)^K \tilde S(-\theta_Z)^{-K}}{1-(1-p) \tilde S(-\theta_Z)^{-1}}, \right),
\end{align*} 
where the fraction $c_Z \tilde S_2(-\theta_Z)^K/\tilde S(-\theta_Z)^K$ is used to
get the workload conditioned on having $K$ type-$2$ jobs in the back, while
the fractions 
 $c_Z \tilde S_2(-\theta_Z)^k \tilde S_1(-\theta_Z)/\tilde S(-\theta_Z)^{k+1}$ are needed
to get the workload conditioned on the fact that the last $k+1$ jobs are
a type-$1$ job followed by $k$ type-$2$ jobs. 

\end{proof}

\begin{theorem}\label{th:R1a}
Let $R^{(1)}$ be the response time distribution of a type-1 job,
then
\begin{align}\label{eq:PR1_th}
P[ R^{(1)} > t] &= P[ W^{(1)} > t] + (1-\lambda)\alpha_1 e^{S_1 t} \e - 
(\nu_1, 0) e^{A^{(1)}t} \colvec{0}{\e} \nonumber \\
&- (\pi_0^{(1)}-\pi_1^{(1)}R^{-1}, 0)  e^{A^{(3)}t}  \colvec{0}{\e} + 
( \pi_0^{(2)}R^{-1}, 0)  e^{A^{(2)} t}  \colvec{0}{\e},
\end{align}
where 
\begin{align*}
A^{(1)}&= \begin{bmatrix}
S^\top \otimes I + (s^* \alpha)^\top \otimes R & (S^\top \otimes I + (s^* \alpha)^\top \otimes R)\xi \alpha_1 \\ 0 & S_1 
\end{bmatrix},\\ 
A^{(2)}&=\begin{bmatrix}
S & (Se)\alpha_1 \\ 0 & S_1
\end{bmatrix}, \\
A^{(3)}&=\begin{bmatrix}
S & S\colvec{e/p}{0}\alpha_1 \\ 0 & S_1
\end{bmatrix}.
\end{align*}
Further, $-\lim_{t\rightarrow \infty} \frac{1}{t}\log P[R^{(1)} > t] = \theta_Z$
and $c_{R^{(1)}} = \lim_{t \rightarrow \infty} e^{\theta_Z t} P[R^{(1)} > t]
= c_{W^{(1)}}\tilde S_1(-\theta_Z)$.
\end{theorem}
\begin{proof}
The result follows from Theorem \ref{th:W1a} and the fact that $P[R^{(1)} > t]$ can
be expressed as
\begin{align*}
P[R^{(1)} > t] = (1-\lambda) &\alpha_1 e^{S_1 t} \e + P[W^{(1)} > t] + \int_0^t \left( -\frac{\partial}{\partial t} P[W^{(1)} > s] \right)
\alpha_1 e^{S_1 (t-s)} \e ds,
\end{align*}
combined with Lemma \ref{lem:cvl}.
\end{proof}

\section{Asymptotic tail improvement ratio}\label{sec:ATIR}

In this section we present results for the asymptotic tail improvement ratio. Most
of the results are expressed in terms of $\tilde S(-\theta_Z)$, $\tilde S_1(-\theta_Z)$,
and $\tilde S_2(-\theta_Z)$, where $\theta_Z$ is the decay rate of the workload $Z$. It is worth noting at this stage that
\begin{align}\label{eq:Stilde}
\tilde S(-\theta_Z) = \frac{\lambda+ \theta_Z}{\lambda},
\end{align}
which follows from \cite[Equation (4)]{abate1} by noting that $\lambda/(\lambda+\theta_Z)$
is the Laplace transform of the inter-arrival time evaluated in $s=\theta_Z$.

The previous theorems yield the following result for the asymptotic tail
improvement ratio:

\begin{theorem}\label{th:ATIR}
The asymptotic tail improvement ratio (ATIR) is equal to
\begin{align}\label{eq:atir}
\mbox{ATIR}(K) &= 1 - \lim_{t \rightarrow \infty} \frac{ P[R_{\mbox{Nudge-K}} > t]}{P[R > t]} \nonumber \\
&=w_1 (\tilde S_2(-\theta_Z)-1)w \frac{1-w^K}{1-w}-(1-w_1)(\tilde S_1(-\theta_Z)-1)(1-(1-p)^K)
\end{align}
with $w_1=p\tilde S_1(-\theta_Z)/\tilde S(-\theta_Z) \in (0,1)$, $w=(1-p)/\tilde S(-\theta_Z)\in (0,1)$ and $w_1+w \in (0,1)$.

Further, the integer $K$ that maximizes ATIR($K$) is given by
\begin{align}\label{eq:Kopt}
K_{opt}=\left\lfloor 
\log \left(\frac{\tilde S_1(-\theta_Z)(\tilde S_2(-\theta_Z)-1)}{\tilde S_2(-\theta_Z)(\tilde S_1(-\theta_Z)-1)}\right)\middle/ \log \tilde S(-\theta_Z) 
\right\rfloor,
\end{align}
if $K_{opt} \geq 0$, otherwise setting $K=0$ is optimal.
\end{theorem}
\begin{proof}
By definition 
\begin{align*}
\mbox{ATIR}(K)&=1 - \frac{p c_{R^{(1)}}+(1-p)c_{R^{(2)}}}{c_{FCFS}}=1 - \frac{p c_{W^{(1)}}\tilde S_1(-\theta_Z)+(1-p)c_{W^{(2)}}\tilde S_2(-\theta_Z)}{c_Z
\tilde S(-\theta_Z)}.
\end{align*}
Combined with
\eqref{eq:cW2} and \eqref{eq:cW1}, this
yields
\begin{align*}
\mbox{ATIR}(K)&=1 -  p \left( w^K
+ w_1 
\frac{1-w^K}{1-w}\right) \frac{\tilde S_1(-\theta_Z)}{\tilde S(-\theta_Z)} \\
&\ \ \ - 
(1-p) 
\left((1-p)^K + (1-(1-p)^K) \tilde S_1(-\theta_Z)\right) \frac{\tilde S_2(-\theta_Z)}{\tilde S(-\theta_Z)},\\
&=(1 - w_1) + w_1 (1-  w^K)
- w_1^2 
\frac{1-w^K}{1-w} \\
&\ \ \ - 
(1-w_1) 
\left((1-p)^K + (1-(1-p)^K) \tilde S_1(-\theta_Z)\right),
\end{align*}
as $1-w_1 = (1-p)\tilde S_2(-\theta_Z)/\tilde S(-\theta_Z)$.
\eqref{eq:atir} now follows by verifying that $1-w-w_1 = w(\tilde S_2(-\theta_Z)-1)$.
To see that $w \in (0,1)$, it suffices
to note that 
$\tilde S(-\theta_Z) > 1$ as $\theta_Z >0$. 
Similarly $\tilde S_i(-\theta_Z) > 1$, which implies that
$w_1 \in (0,1)$ as $\tilde S(-\theta_Z)=p \tilde S_1(-\theta_Z)+(1-p)\tilde S_2(-\theta_Z)$.
Further, $w+w_1 = ((1-p)+p \tilde S_1(-\theta_Z))/\tilde S(-\theta_Z)) < 1$.

Setting the derivative of the ATIR($K$) equal to zero
implies 
\[ \frac{(1-p)^K}{w^K} = \frac{w_1}{1-w_1}\frac{w}{1-w}\frac{\log w}{\log (1-p)}
\frac{\tilde S_2(-\theta_Z)-1}{\tilde S_1(-\theta_Z)-1},\]
which shows that the ATIR($K$) has a unique stationary point.

Define $\Delta_{\mbox{ATIR}}(K)  = \mbox{ATIR}(K+1)-\mbox{ATIR}(K)$.
Recalling that $1-w_1 = (1-p)\tilde S_2(-\theta_Z)/\tilde S(-\theta_Z)$, we have
\begin{align}\label{eq:DeltaATIR}
\Delta_{\mbox{ATIR}}(K)  
&=  w_1 (\tilde S_2(-\theta_Z)-1) w^{K+1} - p(1-w_1)(\tilde S_1(-\theta_Z)-1)
(1-p)^K,\nonumber\\
&= \left( 
\frac{\tilde S_1(-\theta_Z)(\tilde S_2(-\theta_Z)-1)}{\tilde S(-\theta_Z)^{K+2}}
- \frac{\tilde S_2(-\theta_Z)(\tilde S_1(-\theta_Z)-1)}{\tilde S(-\theta_Z)} \right)
p(1-p)^{K+1}.
\end{align}
As $\lim_{K \rightarrow -\infty} \Delta_{\mbox{ATIR}}(K) = +\infty$ and
$\Delta_{\mbox{ATIR}}(K) < 0$ for $K$ sufficiently large, the unique stationary
point of $\mbox{ATIR}(K)$ is a maximum and the optimal integer $K$ is located in the ceil of the unique root of $\Delta_{\mbox{ATIR}}(K)$,
which yields \eqref{eq:Kopt}.
\end{proof}
\paragraph{Remarks:}
The expression in \eqref{eq:DeltaATIR}
shows that the increase in the ATIR($K$) decreases with $K$ as long as 
it remains positive. Thus the gain obtained by increasing $K$ by one decreases with $K$
until the optimal $K$ is reached. 

It is worth noting that if the type-$i$ jobs have an exponential job size distribution with
parameter $\mu_i$, then \eqref{eq:Kopt} simplifies to
\begin{align}\label{eq:Kopt_expo}
K_{opt} = \lfloor \log(\mu_1 / \mu_2) / \log(\tilde S(-\theta_Z)) \rfloor,
\end{align}
which implies that $K_{opt}$ is non-decreasing in $\lambda$ when the job sizes are exponential 
and $\mu_1/\mu_2 = E[X_2]/E[X_1]>1$ (as 
$\tilde S(-\theta_Z)$ decreases in $\lambda$). As we will demonstrate in Section
\ref{sec:num}, this property does not necessarily hold when the type-$2$ jobs are no longer
exponential even if $X_2$ stochastically dominates $X_1$.

\begin{theorem}\label{th:ATIRcond}
The ATIR$(K)> 0$ if and only if
\begin{align}\label{eq:K}
\left. \left(1 - \frac{1}{\tilde S_2(-\theta_Z)}\right)
 \middle/ \left(1 - \frac{1}{\tilde S_1(-\theta_Z)}\right) \right. 
 > \left(1+\frac{\theta_Z}{\lambda p}\right)
\frac{1-(1-p)^K}{1-\left(\frac{\lambda(1-p)}{\lambda+\theta_Z}\right)^K},
\end{align}
meaning the ATIR$(1)> 0$ if and only if
\begin{align}\label{eq:K1}
\left. \left(1 - \frac{1}{\tilde S_2(-\theta_Z)}\right)
\middle/ \left(1 - \frac{1}{\tilde S_1(-\theta_Z)}\right )\right. 
> 1+\frac{\theta_Z}{\lambda},
\end{align}
and the \mbox{ATIR}$(K) > 0$ for any $K$ if and only if
\begin{align}\label{eq:allK}
\left. \left(1 - \frac{1}{\tilde S_2(-\theta_Z)}\right)
\middle/ \left(1 - \frac{1}{\tilde S_1(-\theta_Z)}\right )\right.
> 1+\frac{\theta_Z}{\lambda p}.
\end{align}
\end{theorem}
\begin{proof}
From \eqref{eq:atir} and the definition of $w_1$, we have ATIR($K)>0$ if and only if
\begin{align*}
p \tilde S_1(-\theta_Z) &(\tilde S_2(-\theta_Z)-1) w \frac{1-w^K}{1-w}
> (1-p)\tilde S_2(-\theta_Z) (\tilde S_1(-\theta_Z)-1) (1-(1-p)^K).
\end{align*}
This condition can be restated as
\begin{align*}
\left. \left(1 - \frac{1}{\tilde S_2(-\theta_Z)}\right)
\middle/ \left(1 - \frac{1}{\tilde S_1(-\theta_Z)}\right )\right.
> \frac{(1-w)\tilde S(-\theta_Z)}{p} \frac{1-(1-p)^K}{1-w^K}.
\end{align*}
Using $w=(1-p)/\tilde S(-\theta_Z)$ and \eqref{eq:Stilde} we obtain \eqref{eq:K}.
Setting $K=1$ and taking the limit for $K$ to infinity yield \eqref{eq:K1} and
\eqref{eq:allK}.
The result for  $K=1$ is also immediate from \eqref{eq:DeltaATIR} as
$\mbox{ATIR}(1) = \Delta_{\mbox{ATIR}}(0)$.


\end{proof}

\paragraph{Remarks:} The condition in \eqref{eq:K1} is very similar to Theorem 4.3 in \cite{nudge}. Moreover when
$K=1$, we have
\begin{align*}
\mbox{ATIR}(1) &= p(1-p)\left(\frac{\tilde S_1(-\theta_Z)}{\tilde S(-\theta_Z)}
\frac{\tilde S_2(-\theta_Z)-1}{\tilde S(-\theta_Z)}
-\frac{\tilde S_2(-\theta_Z)}{\tilde S(-\theta_Z)}(\tilde S_1(-\theta_Z)-1) \right)\\
&=\frac{p(1-p)}{\tilde S(-\theta_Z)}
\left(\tilde S_2(-\theta_Z)-\frac{\tilde S_1(-\theta_Z)}{\tilde S(-\theta_Z)}
-\tilde S_1(-\theta_Z)\tilde S_2(-\theta_Z)(1-\tilde S(-\theta_Z)^{-1}) \right)\\
&=\frac{\lambda p(1-p)}{\lambda+\theta_Z}
\left(\tilde S_2(-\theta_Z)-\frac{\lambda}{\lambda+\theta_Z}\tilde S_1(-\theta_Z)-
\frac{\theta_Z}{\lambda+\theta_Z}\tilde S_1(-\theta_Z)\tilde S_2(-\theta_Z)\right),
\end{align*}
which is again similar in form to Theorem 4.3 in \cite{nudge}.

When type-$i$ jobs have an exponential distribution with mean $1/\mu_i$, for $i=1,2$,
we have $\tilde S_i(s)=\mu_i/(\mu_i+s)$ and \eqref{eq:K1} simplifies to
\[\frac{\mu_1}{\mu_2} > \tilde S(-\theta_Z) = 1 + \theta_Z/\lambda.\]
As $\lim_{\lambda \rightarrow 1^-} \tilde S(-\theta_Z) = 1$, this means that
for $\mu_1/\mu_2 =E[X_2]/E[X_1]>1$, there exists a $\bar \lambda \in (0,1)$ such that
ATIR$(1)>0$ for $\lambda > \bar \lambda$. The next theorem shows that
this result holds in general for phase-type distributions:

\begin{theorem}\label{th:KoptHT}
When $E[X_2]>E[X_1]$, then ATIR$(K)>0$ for $\lambda$ sufficiently close to one
for any $K$. Further, if $E[X_2]>E[X_1]$, then $\lim_{\lambda \rightarrow 1^-} K_{opt}=
\infty$.
\end{theorem}
\begin{proof}
Using the Taylor series expansion of the exponential function, we readily see that
\begin{align}\label{eq:Sexpan}
\tilde S_i(-\theta_Z) = \sum_{k=0}^\infty \frac{\theta_Z^k E[X_i^k]}{k!}= 1 + \theta_Z E[X_i] + o(\theta_Z^2),
\end{align}
for $i=1,2$ and similarly $\tilde S(-\theta_Z) = 1 + \theta_Z E[X] + o(\theta_Z^2)$.
This implies that
\begin{align}\label{eq:limratio}
\lim_{\lambda \rightarrow 1^-}
\left. \left( 1-\frac{1}{\tilde S_2(-\theta_Z)}\right) \middle/ \left(
1-\frac{1}{\tilde S_1(-\theta_Z)}\right)\right.=
\frac{E[X_2]}{E[X_1]}.
\end{align}
as $\theta_Z$ tends to zero when $\lambda$ tends to one. 
Hence, if $E[X_2]/E[X]>1$, then ATIR$(K)>0$ for any $K$ for $\lambda$ close enough to $1$ due to \eqref{eq:allK} as $1+\theta_Z/\lambda$ tends to one.

To determine $\lim_{\lambda \rightarrow 1^-} K_{opt}$ we note that \eqref{eq:limratio}
implies that
\begin{align}\label{eq:limlam1}
\lim_{\lambda \rightarrow 1^-} K_{opt}= 
\lim_{\lambda \rightarrow 1^-} \left\lfloor  \frac{\log(E[X_2]/E[X_1])}{\log(\tilde S(-\theta_Z))}
\right\rfloor.
\end{align}
This proves that $K_{opt}$ tends to $\infty$ when $E[X_2]/E[X_1] > 1$, while
$K_{opt}$ tends to $-\infty$ if $E[X_2]/E[X_1] < 1$.
\end{proof}

\begin{theorem}\label{th:Kapprox}
For $\lambda$ close to one and $E[X_2]>E[X_1]$, we have
\begin{align}\label{eq:Kapprox}
K_{opt} \approx \left\lfloor 
\frac{\log (E[X_2]/E[X_1]) E[X^2]}{2(1-\lambda)} 
\right\rfloor \approx \left\lfloor 
\log \left(\frac{E[X_2]}{E[X_1]}\right) E[Z]
\right\rfloor,
\end{align}
where $\log()$ is the natural logarithm.
\end{theorem}
\begin{proof}
Using \eqref{eq:limlam1} and the fact that
$\tilde S(-\theta_Z) = 1 + \theta_Z + o(\theta_Z^2)$ (as $E[X]=1$), we
have for $\lambda$ close to one
\[K_{opt} \approx \left\lfloor  \frac{\log(E[X_2]/E[X_1])}{\log(1+\theta_Z)} \right\rfloor.\]
The result therefore
follows from the classic heavy-traffic limit of Kingman \cite{kingman62} for the GI/G/1 queue (when
both the inter-arrival time $I$ and service time distribution $X$ has  finite variance), which
states that the decay rate $\theta_Z$ of the waiting time in the heavy traffic limit equals
\begin{align}\label{eq:thetaHT}
\frac{2 \left( E[I]-E[X]\right)}{Var[I]+Var[X]} =\frac{2(\frac{1}{\lambda}-1)}{\frac{1}{\lambda^2} + 
pE[X_1^2]+(1-p)E[X_2^2] -1},
\end{align}
and the Taylor series expansion of $\log(1+x) = \sum_{n=1}^\infty (-1)^{n+1}x^n/n$.
The expression using $E[Z]$ is due to the fact that 
$E[Z] = \lambda E[X^2]/(2(1-\lambda))$.
\end{proof}

\begin{figure*}[t!]
  \centering
  \includegraphics[width=0.55\linewidth]{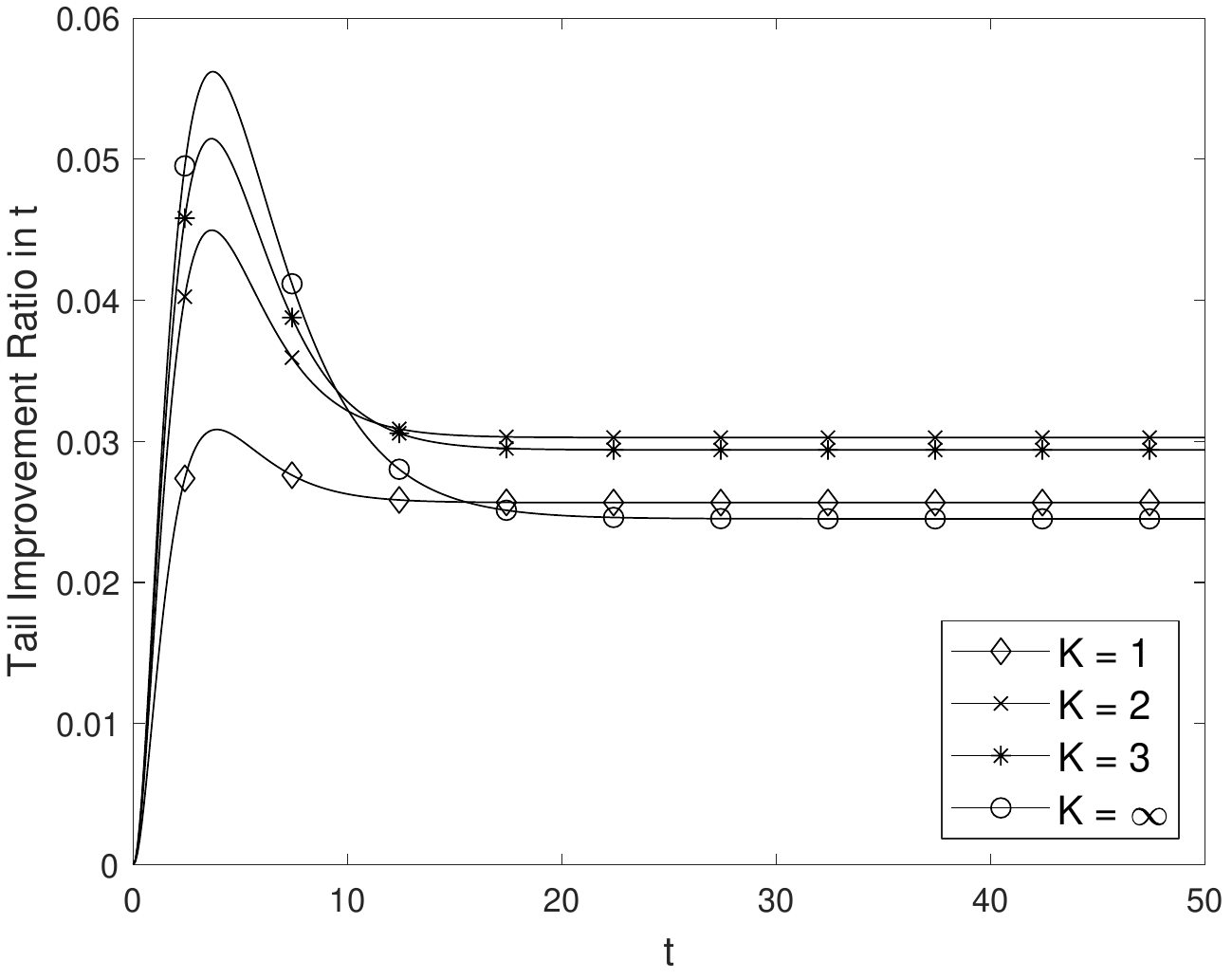}
\caption{The tail improvement ratio of Nudge-K over FCFS for expo jobs with
 $E[X_2]/E[X_1] = 2, \lambda = 3/4$ and $p=1/2$.}
\label{fig:TIR1}
\end{figure*}

\section{Numerical Results and Insights}\label{sec:num}

In this section we present various new insights on stochastic and asymptotic improvements
of FCFS. We start with the stochastic improvements.

\subsection{On stochastic improvements}
Figure \ref{fig:TIR1} plots the tail improvement ratio of Nudge-K over FCFS 
for $\lambda = 3/4$ when both type-$1$ and type-$2$ jobs have an exponential distribution with
 $E[X_2]/E[X_1] = 2$ and $p=1/2$. A number of observations can be
 made:
 \begin{enumerate}
 \item Nudge-$K$ stochastically improves FCFS for $K=1,2,3$ and $\infty$, even though
 type-$1$ jobs are not necessarily smaller than type-$2$ jobs. This illustrates that
 FCFS can be stochastically improved upon under far weaker conditions than in \cite{nudge}.
 \item Nudge-$2$ and Nudge-$3$ both stochastically improve Nudge-$1$, but neither 
 stochastically improves the other. From \eqref{eq:Kopt} we also know that
 setting $K=2$ optimizes the ATIR$(K)$. This implies that there does not exist a
 $K$ that minimizes $P[R_{\mbox{Nudge}-K} > t]$ for all $t$.
 \item Setting $K=\infty$ is best for reducing $P[R_{\mbox{Nudge-}K} > t]$ for small
 $t$, but does not stochastically improve Nudge-$K$ for $K \in \{1,2,3\}$.  
 \end{enumerate}

Figure \ref{fig:TIR2} plots the tail improvement ratio of Nudge-K over FCFS 
for the same setting as Figure \ref{fig:TIR1}, except that $E[X_2]/E[X_1]=3/2$.
The main observation in this plot is that while $K=1$ and $2$ results in a 
stochastic improvement over FCFS, setting $K\geq 3$ does not
(as the ATIR($K$) decreases in $K$ beyond $K_{opt}$).
In other words, {\it in some cases the stochastic improvement over FCFS 
can be lost if we allow that type-$1$ jobs can pass too many
type-$2$ jobs}.

Figure \ref{fig:TIR3} considers the scenario with
$p=0.7$, $\lambda = 0.7$, $E[X_2]/E[X_1]=1.2$. Type-$1$ jobs are exponential, while
the type-$2$ jobs follow an order-$2$ hyper-exponential distribution 
with $SCV=2$ and shape parameter $f=9/10$.  This means that type-$2$ jobs are
a mixture of two classes of exponential jobs: one with mean
$1/\mu_{21}\approx 1.034$ and one with mean $1/\mu_{22} \approx 7.674$, where $10\%$ of the workload is offered by the jobs belonging to the class with the larger mean
(that is, $\alpha_2 \approx(0.985,0.015)$). 
It is important to note that both $1/\mu_{21}$ and $1/\mu_{22}$ are larger
than $1$, while the mean of the type-$1$ jobs is $100/106 < 1$. As a result 
$X_2$ stochastically dominates $X_1$ (in the first order), that is
\[P[X_2 > t] > P[X_1 > t],\]
for any $t>0$. The plot shows that while all the considered $K$ values 
result in an {\it asymptotic} tail improvement ratio, none of them stochastically
improves FCFS. This example shows that {\it having an asymptotic improvement does
not imply a stochastic improvement in general even if $X_2$ stochastically
dominates $X_1$}. The intuition is that while it is good to swap type-$1$ jobs
with the type-$2$ jobs with mean $1/\mu_{22}$, the swaps with the jobs with
mean $1/\mu_{21}$ are not beneficial as their mean is fairly close to the mean of the
type-$1$ jobs.

 \begin{figure*}[t!]
  \centering
  \includegraphics[width=0.55\linewidth]{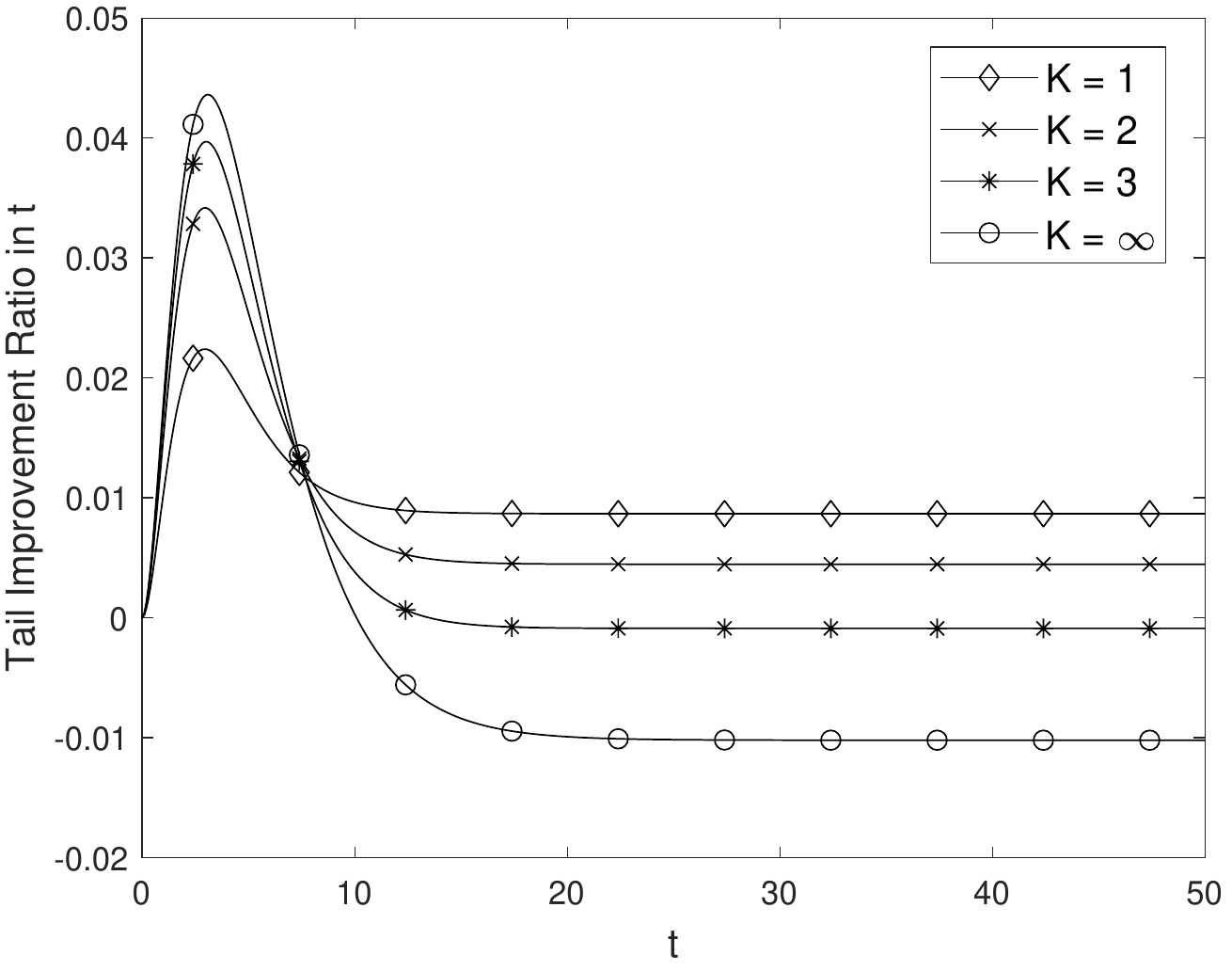}
\caption{The tail improvement ratio of Nudge-K over FCFS for expo jobs with
 $E[X_2]/E[X_1] = 3/2, \lambda = 3/4$ and $p=1/2$.}
\label{fig:TIR2}
\end{figure*}

 \begin{figure*}[t!]
  \centering
  \includegraphics[width=0.55\linewidth]{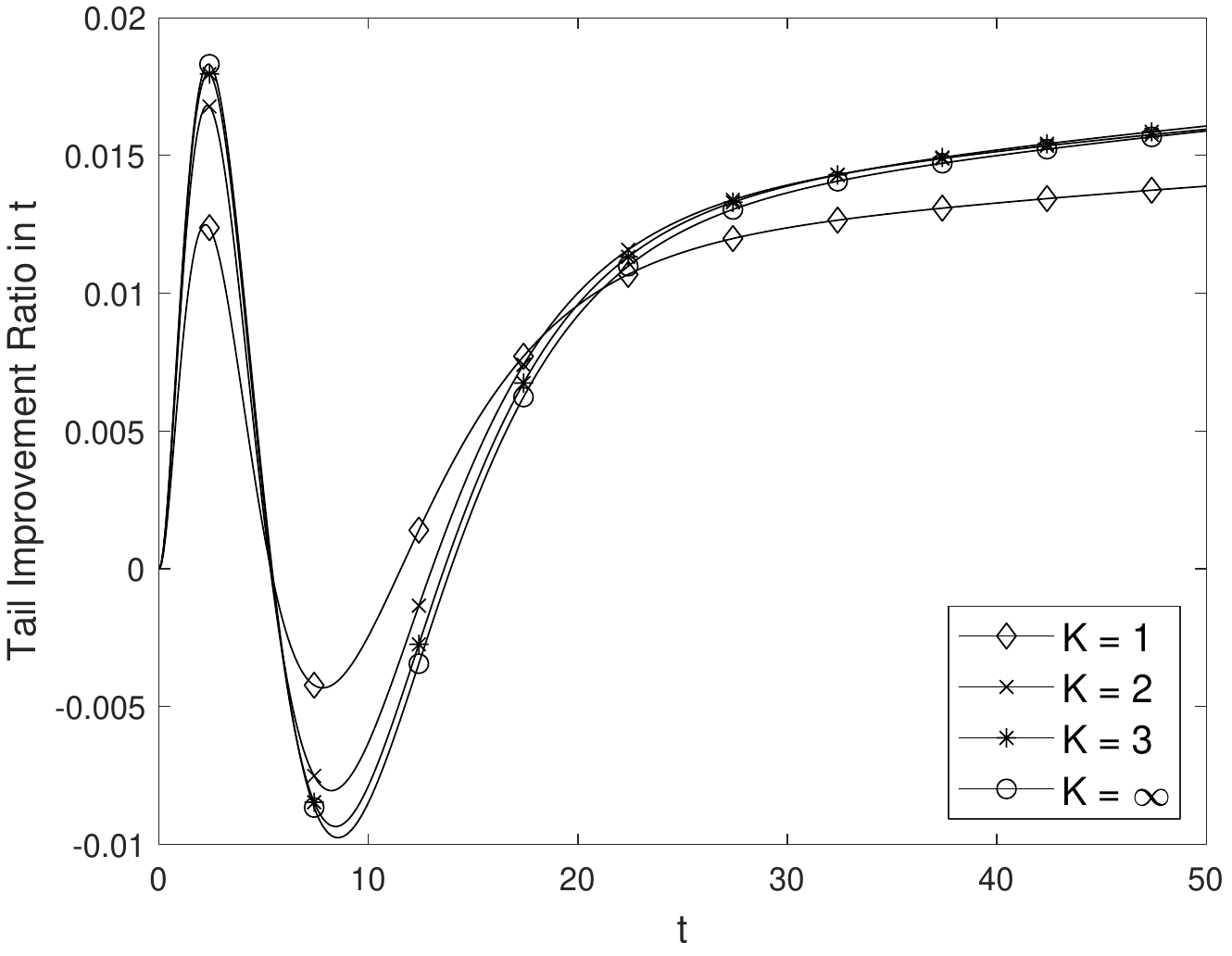}
\caption{The tail improvement ratio of Nudge-K over FCFS for expo type-$1$ and
H2 ($SCV=2, f=9/10$) type-2 jobs with
 $E[X_2]/E[X_1] = 1.2, \lambda = 0.7$ and $p=0.7$.}
\label{fig:TIR3}
\end{figure*}

%

\subsection{On asymptotic tail improvements}

In this subsection we address two issues. First, we noted that if both the
type-$1$ and type-$2$ jobs are exponential, then $K_{opt}$ is non-decreasing in
$\lambda$ if $E[X_2]>E[X_1]$ (see \eqref{eq:Kopt_expo}). We now demonstrate that if we
make the type-$2$ jobs hyper-exponential, this is not necessarily the case
even if $X_2$ stochastically
dominates $X_1$.
Second, for exponential job sizes having $E[X_1] > E[X_2]$ implies that
$K_{opt}=0$, meaning we cannot asymptotically improve upon FCFS. 
We illustrate that when type-$2$ jobs are no longer exponential and 
$E[X_1] > E[X_2]$, we can still achieve
an asymptotic tail improvement in some cases. 

The scenario considered is similar to Figure \ref{fig:TIR3}, that is, $p=0.7$,
type-$1$ jobs are exponential and type-$2$ jobs are hyper-exponential with
$SCV=2$ and shape parameter $f=0.9$. We vary $\lambda$ from $0.01$ to $0.99$
and $E[X_2]/E[X_1]$ from $0.4$ to $2$ (in Figure \ref{fig:TIR3} these were
set at $0.7$ and $1.2$, respectively). Figure \ref{fig:ATIR} presents
two contour plots: one for the ATIR($K_{opt})$ with contour lines from $0.03$ to $0.15$
in steps of $0.03$ and one for $K_{opt}$ with contour lines in $1,2,\ldots,20$.

\begin{figure*}[t!]

\begin{subfigure}{.46\textwidth}
  \centering
  \includegraphics[width=1\linewidth]{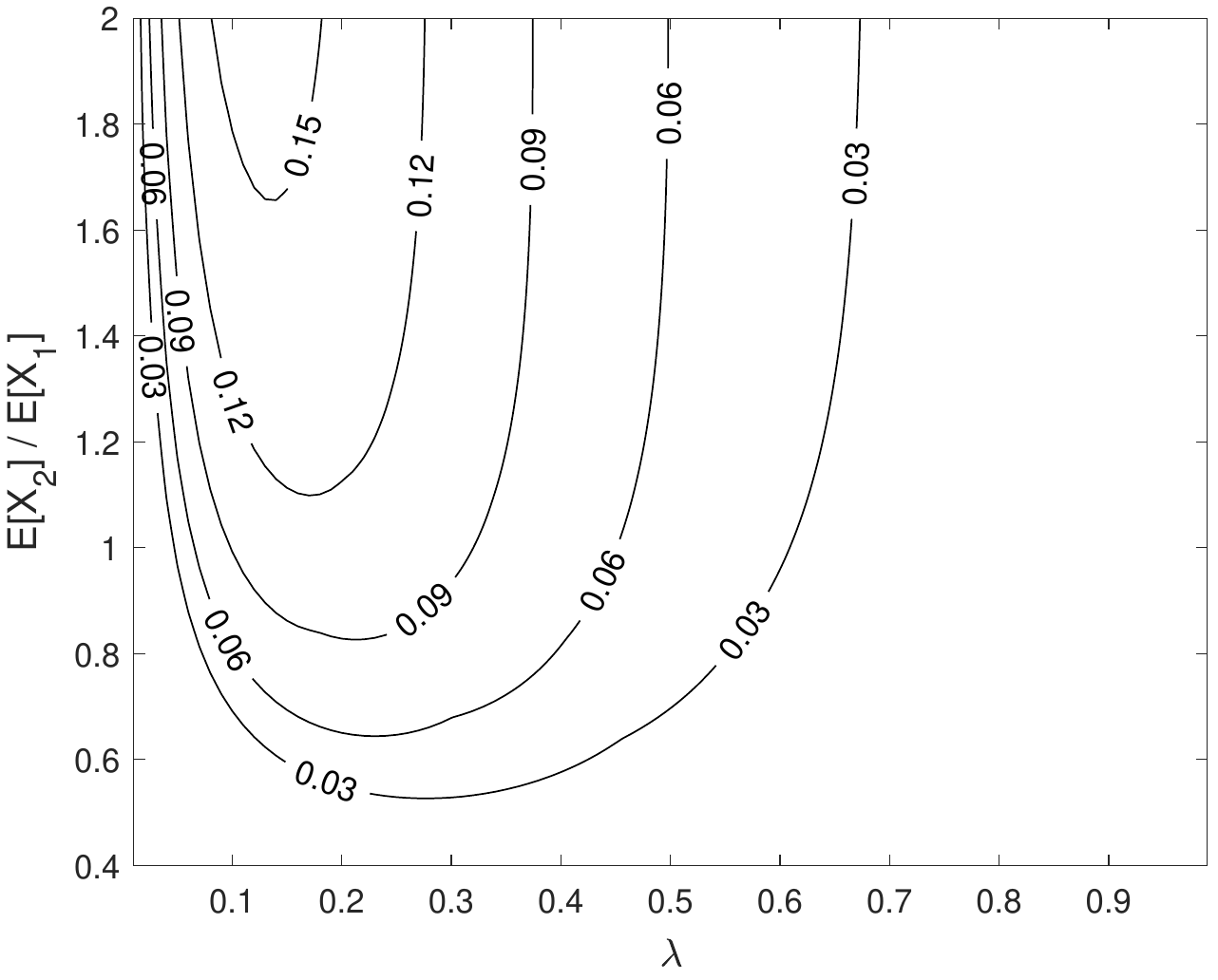}
	\caption{Contour plot of ATIR($K_{opt}$)}
	\label{fig:ExpH2Kopt_p07}
\end{subfigure}
\begin{subfigure}{.46\textwidth}
  \centering
  \includegraphics[width=1\linewidth]{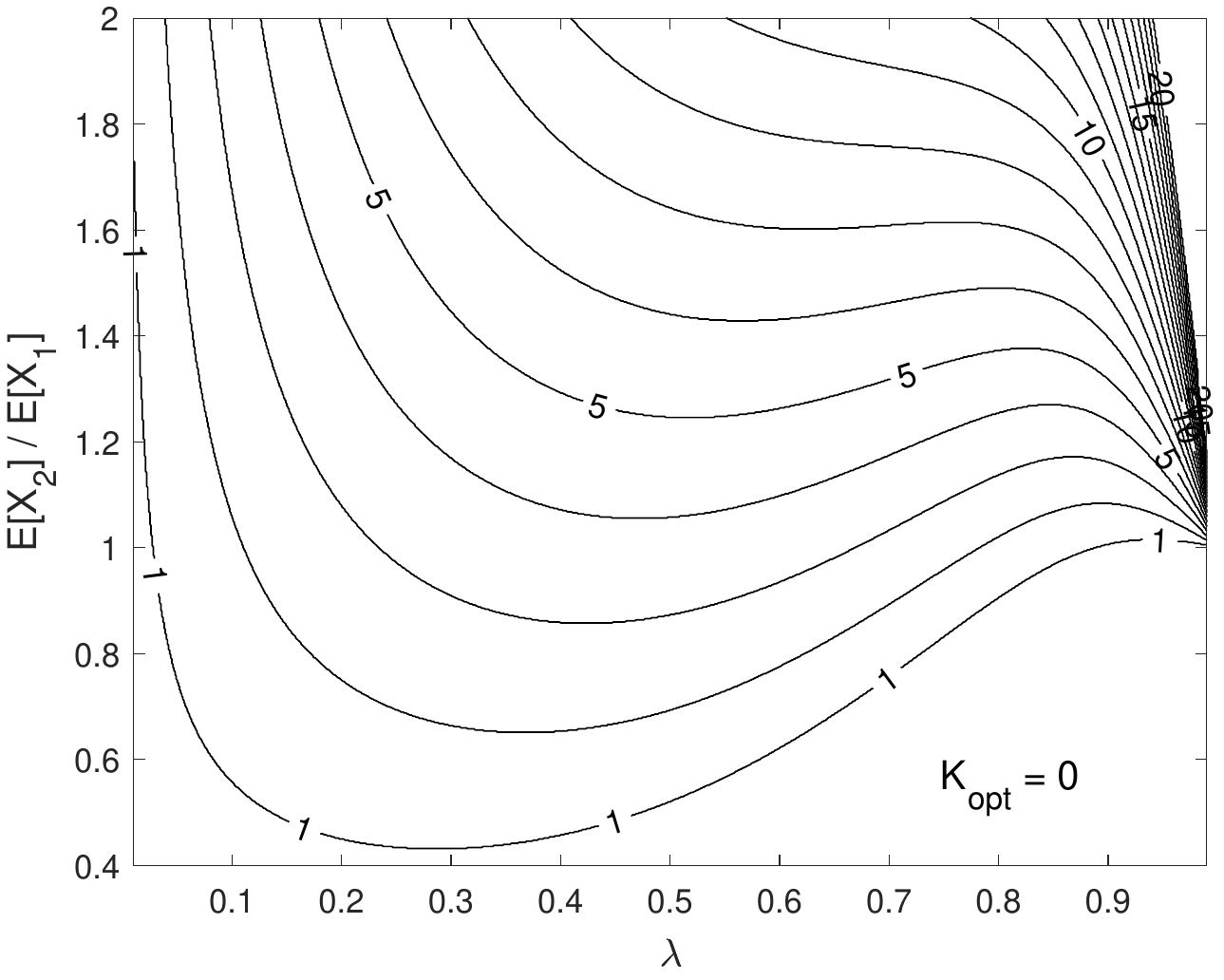}
	\caption{Contour plot of $K_{opt}$ (bounded by $20$)}
	\label{fig:ExpH2best_p07}
\end{subfigure}
\caption{The asymptotic tail improvement ratio of Nudge-K over FCFS (a)
and optimal $K$ (b) for $p=0.7$ with
exponential type-1 and hyper-exponential type-2 jobs ($SCV = 2$ and $f=9/10$).}
\label{fig:ATIR}
\end{figure*}

Looking at the region where $E[X_2]/E[X_1]<1$ clearly shows that we can have an
asymptotic tail improvement even when $E[X_1]>E[X_2]$. If we focus on the
line with $E[X_2]/E[X_1]=1.2$ in the contour plot of $K_{opt}$, we note that
$K_{opt}$ first increases to $4$, then drops to $3$ and finally starts to increase
(without bound due to Theorem \ref{th:KoptHT}) as $\lambda$ tends to one.
This shows that $K_{opt}$ can decrease as a function of $\lambda$. Further
note that when $E[X_1]/E[X_2]=1.2$, then $X_2$ stochastically dominates
$X_1$ (as explained when discussing Figure \ref{fig:TIR3}).

\section{Conclusions}\label{sec:concl}
In this paper we demonstrated that the First-Come-First-Served scheduling algorithm can
be stochastically improved upon by the Nudge-$K$ algorithm under far weaker conditions
that the ones considered in \cite{nudge}. This is practically relevant as it indicates that
it may suffice to identify certain job types, where jobs belonging to one type are typically larger than jobs belonging to another, in order to improve all of the response time percentiles of First-Come-First-Served scheduling. We did this by deriving explicit expressions for the response time
of Nudge-$K$ for the system defined in Section \ref{sec:system}. In addition we presented
a number of elegant results on the asymptotic tail improvement ratio, such as the
expression for the optimal $K$. We also presented various new insights on stochastic
and asymptotic improvements upon First-Come-First-Served scheduling.

The current results can be extended in a number of ways. It is not too difficult to
include a third job type that cannot be swapped with either type-$1$ or type-$2$ jobs,
though we expect smaller gains in such a scenario and the notations become somewhat heavier. The results on the ATIR($K$)
presented in Section \ref{sec:ATIR} and their proofs remain mostly valid if we relax the assumptions that type-$1$ and type-$2$ job sizes are phase-type and simply
demand that $X=pX_1+(1-p)X_2$ is phase-type. In such case we could also define the
type-$1$ and type-$2$ jobs using the job size (as in \cite{nudge}). 

It is also worthwhile seeing what happens if we consider a larger class of Nudge-like
policies, for instance, suppose that we allow that type-$2$ jobs can be involved
in at most $L$ swaps instead of just one. When $L > 1$, the analysis performed for
the type-$1$ jobs in Section \ref{sec:t1} fails as we can no longer make a similar connection
with the FCFS queue. We believe that in such case it is possible to rely on the
framework of Markov modulated fluid queues to study the response time of a type-$1$ job.
In fact initially we used this approach for the analysis of the response
time of a type-$1$ job of Nudge-$K$ before coming up with the more elegant approach presented in Section
\ref{sec:t1}. We note that both these approaches yielded the same numerical results.

\section*{Acknowledgement}
The author likes to thank Isaac Grosof for some useful discussions and suggestions.

\bibliographystyle{plain}
\bibliography{../../PhD/thesis}

\newpage

\appendix

\section{Sensitivity analysis of the ATIR($K$) and $K_{opt}$}

In this appendix we look at the impact of the arrival rate $\lambda$, the
fraction of type-$1$ jobs $p$, the ratio of the mean job sizes $E[X_2]/E[X_1]$
as well as the job size distribution on the ATIR($K$) and the optimal $K_{opt}$.
We first consider the case where both the type-$1$ and type-$2$ job size 
distributions are exponential
(see Figure \ref{fig:ATIR_expo}), 
then we assume both are hyperexponential with an SCV equal to $5$ and balanced means
(see Figure \ref{fig:ATIR_h2}). Finally, we look at the case where both distributions
are Erlang-$5$ distributions (see Figure \ref{fig:ATIR_erl}), 
which means the SCV equals $1/5$.

The main observations regarding these plots are as follows:
\begin{enumerate}
\item The parameter ranges for which the ATIR$(1) > 0$ tend to increase as the
job size variability increases (compare subfigures (a) and (b)
in Figures \ref{fig:ATIR_expo} to \ref{fig:ATIR_erl}). Note that in these
experiments also the variability of the type-$1$ jobs increases. 
Further remark that Theorem \ref{th:ATIRcond} implies that if ATIR($K) > 0$ for
some $K$, then ATIR($1) > 0$.
\item By comparing subfigures (a) and (b) with (c) and (d) in Figures \ref{fig:ATIR_expo} to \ref{fig:ATIR_erl}, we note that setting $K$
equal to the optimal $K$ only gives a limited gain over setting $K=1$. This is in agreement
with our earlier observation (see the remark after Theorem \ref{th:ATIR}) that each time that we increase $K$ by one (until $K=K_{opt}$), 
we obtain a smaller gain in the ATIR($K$).
\item When the job size variability increases we observe larger gains (exceeding $0.2$ for 
the hyperexpontial case, $0.12$ for exponential jobs and $0.08$ for Erlang jobs), but the
gains tend to
occur at lower system loads $\lambda$ when the variability increases (see subfigures (a) to (d)). 
\item The optimal $K$ tends to grow with $\lambda$ and $p$ (though this is not always the
case as seen in Figure \ref{fig:ATIR}(b)) and increases significantly as the variability in the job size
increases (compare subfigures (e) and (f) in Figures \ref{fig:ATIR_expo} to \ref{fig:ATIR_erl}). 
\item While the ATIR($K$) always becomes positive under heavy traffic and $K_{opt}$ tends
to infinity, the value of the ATIR($K$) tends to zero as $\lambda$ tends to one. This can be understood by noting that the relative gain of limited swapping between type-$1$ and 
type-$2$ jobs vanishes as the mean workload tends to infinity. 
\end{enumerate}

\begin{figure*}[h!]
\begin{subfigure}{.46\textwidth}
  \centering
  \includegraphics[width=1\linewidth]{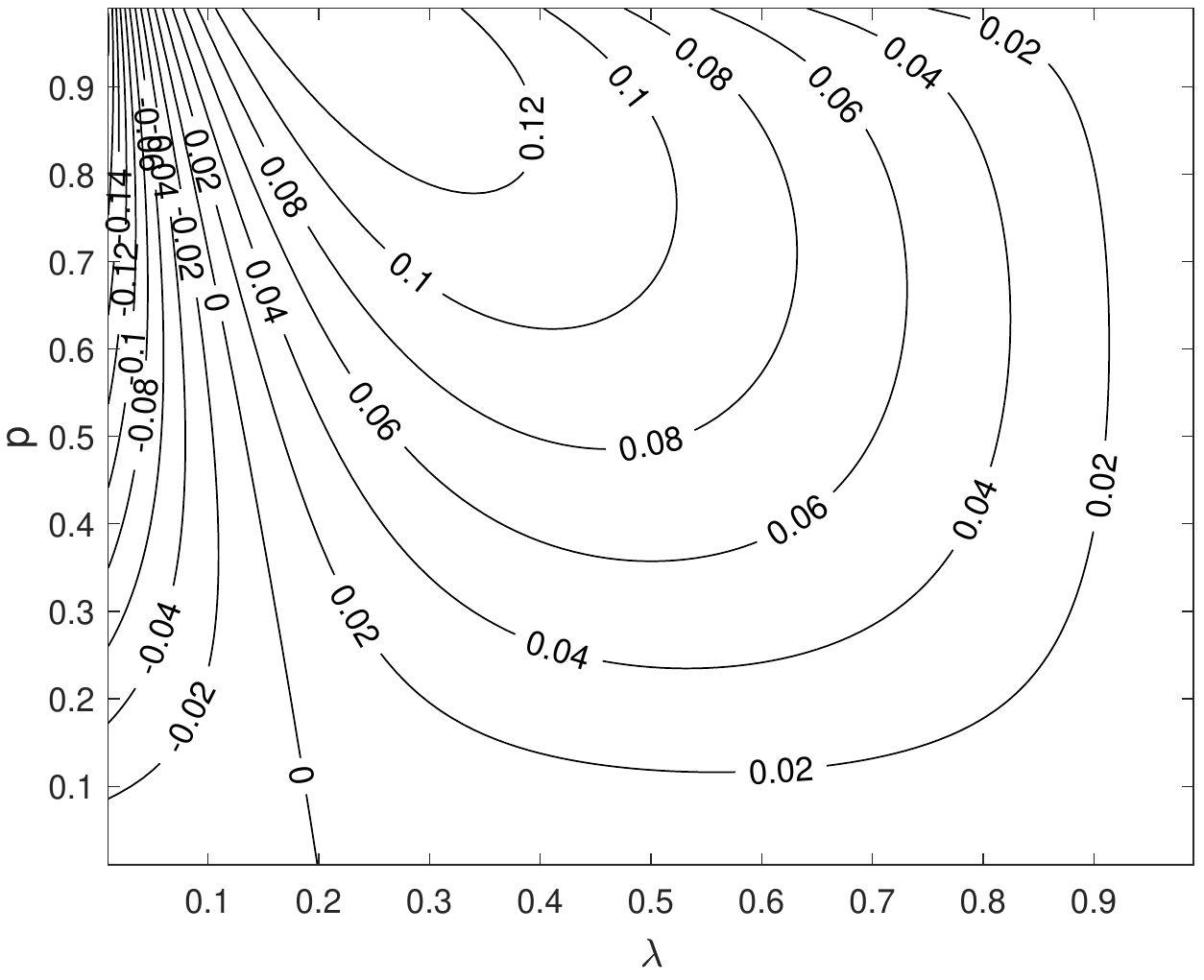}
	\caption{$K = 1$, $E[X_2]/E[X_1] = 5$}
	\label{fig:K1r5}
\end{subfigure}
\begin{subfigure}{.46\textwidth}
  \centering
  \includegraphics[width=1\linewidth]{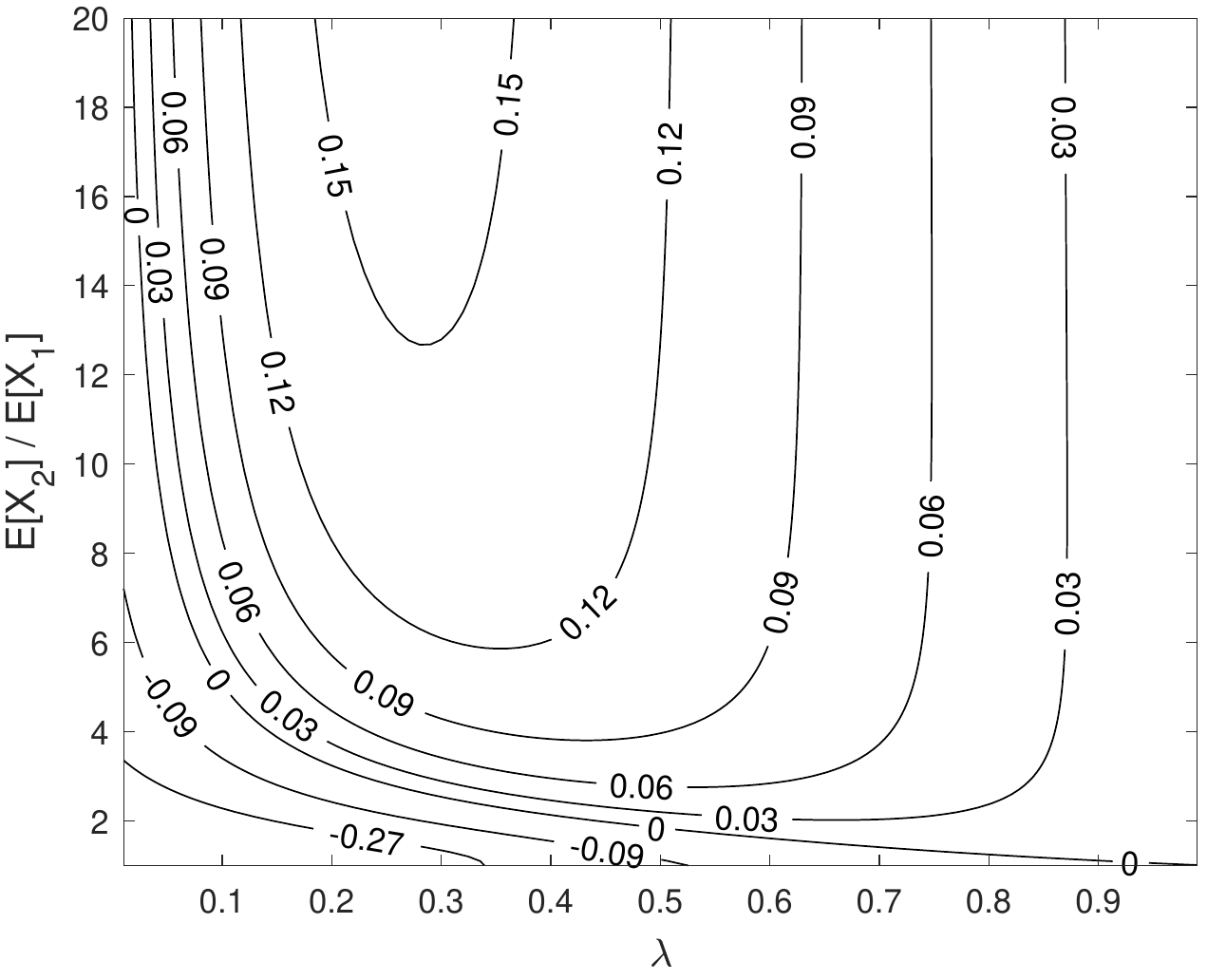}
	\caption{$K = 1$, $p = 0.7$}
	\label{fig:K1p07}
\end{subfigure}
\begin{subfigure}{.46\textwidth}
  \centering
  \includegraphics[width=1\linewidth]{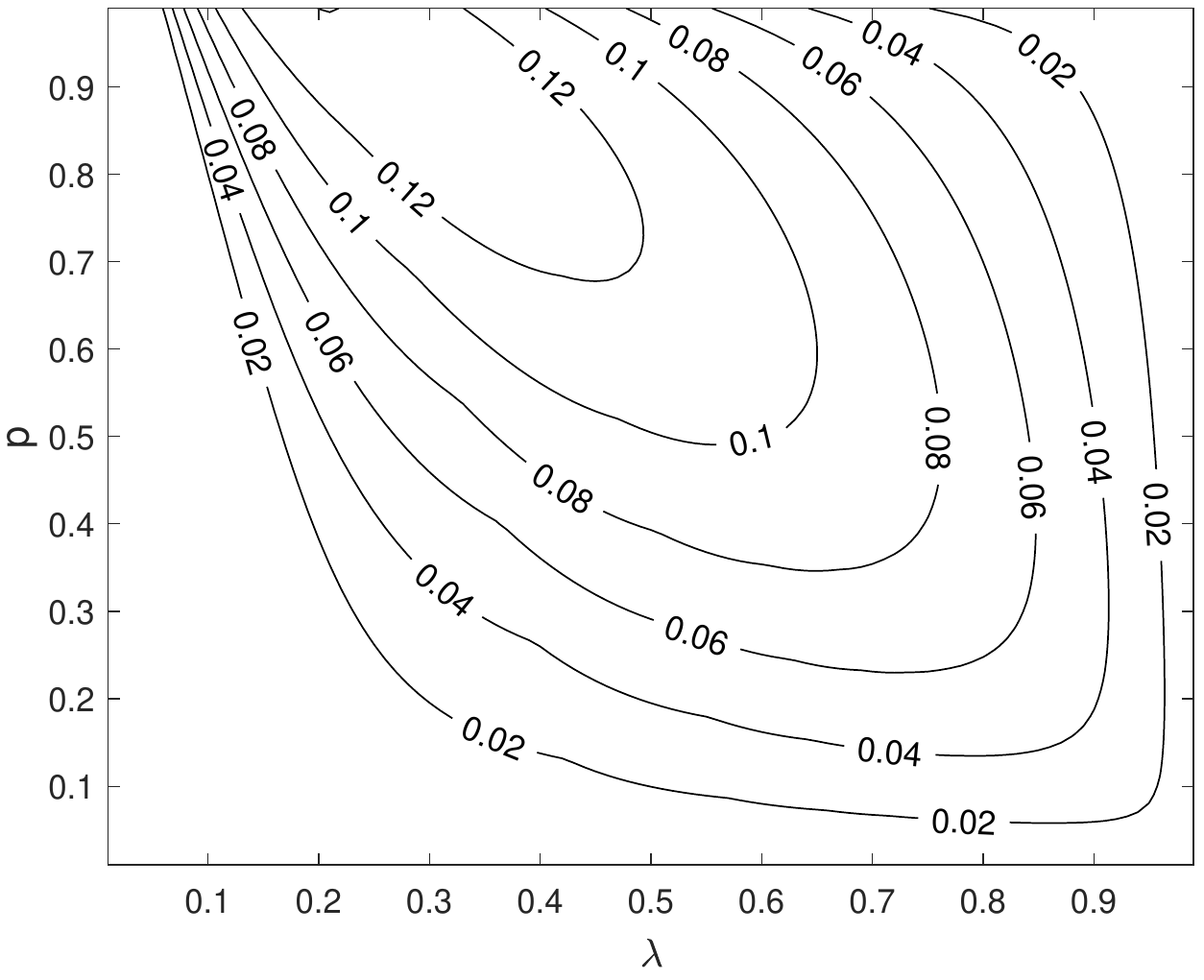}
	\caption{$K = K_{opt}$, $E[X_2]/E[X_1] =5$}
	\label{fig:Kopt_r5}
\end{subfigure}
\begin{subfigure}{.46\textwidth}
  \centering
  \includegraphics[width=1\linewidth]{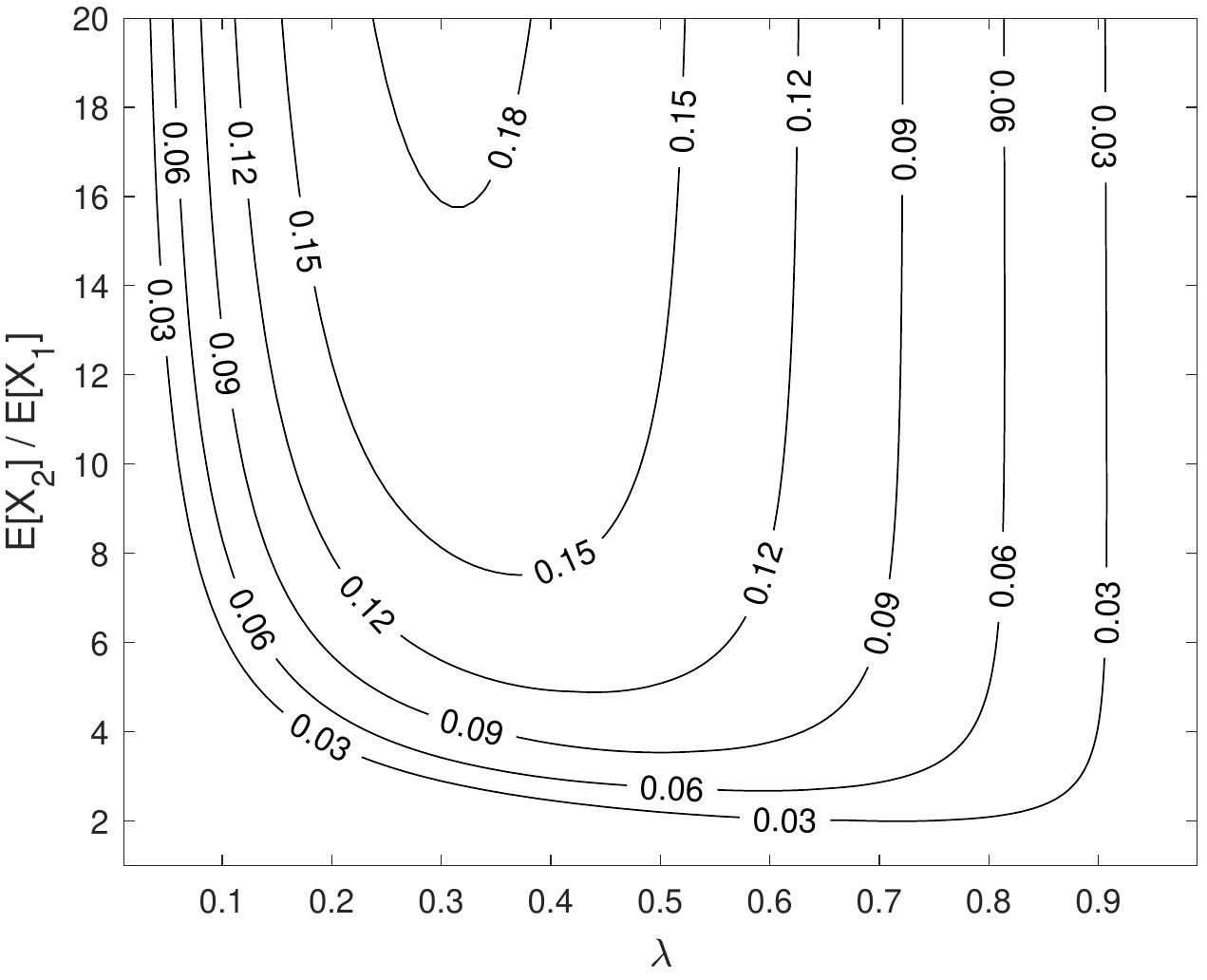}
	\caption{$K = K_{opt}$, $p = 0.7$}
	\label{fig:Kopt_p07}
\end{subfigure}
\begin{subfigure}{.46\textwidth}
  \centering
  \includegraphics[width=1\linewidth]{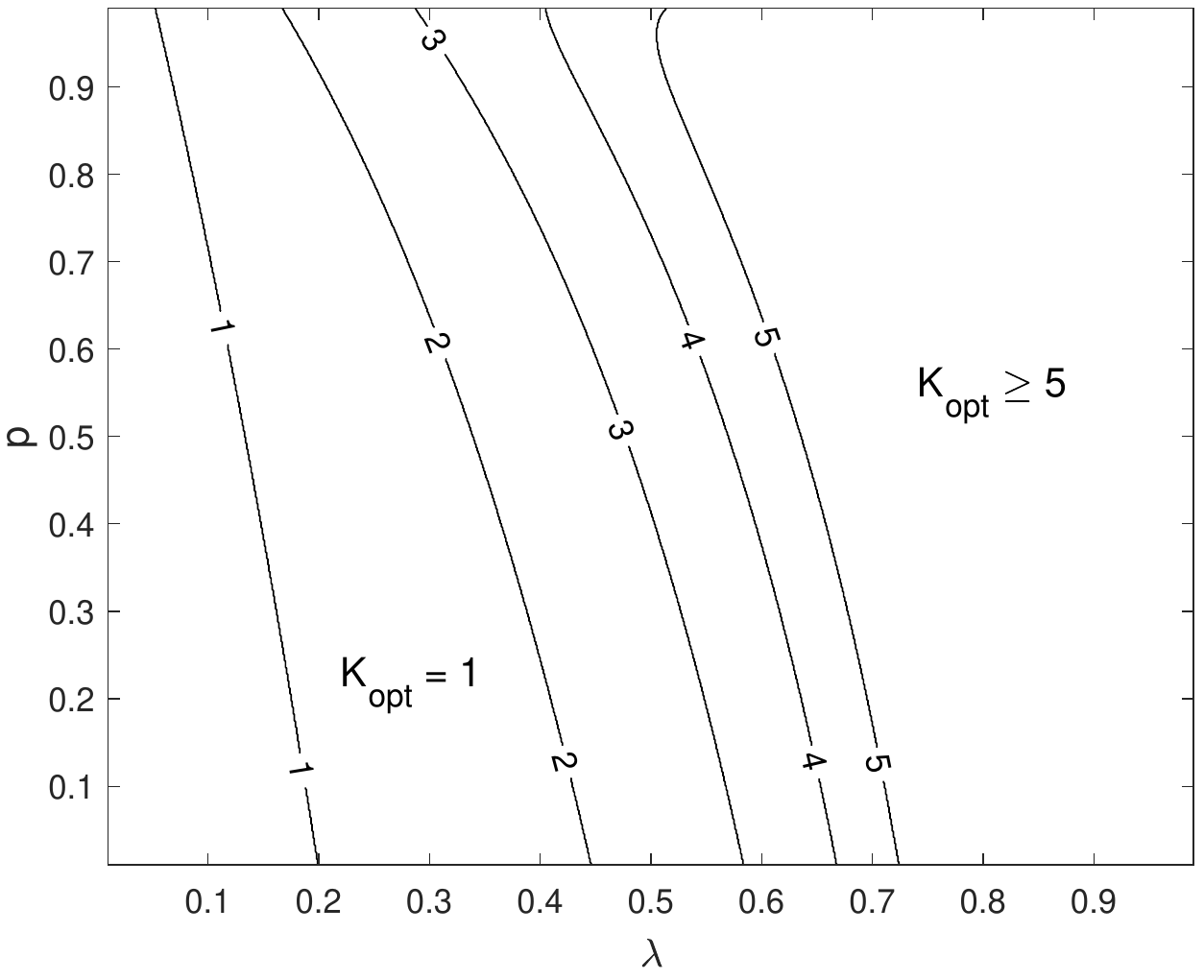}
	\caption{$E[X_2]/E[X_1] = 5$}
	\label{fig:best_r5}
\end{subfigure}
\hspace*{0.35cm}
\begin{subfigure}{.46\textwidth}
  \centering
  \includegraphics[width=1\linewidth]{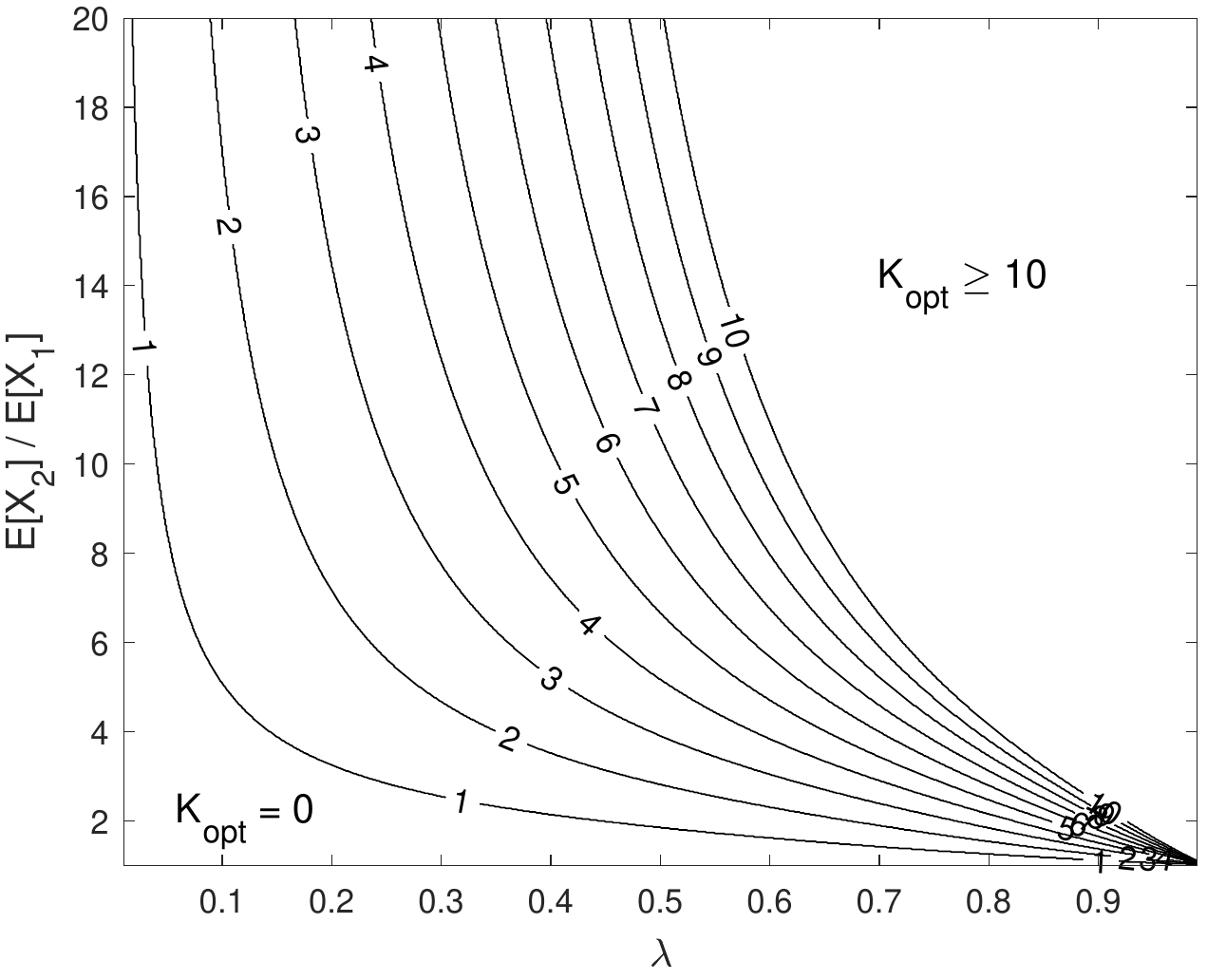}
	\caption{$p=0.7$}
	\label{fig:best_p07}
\end{subfigure}
\caption{The asymptotic tail improvement ratio of Nudge-K over FCFS (a)-(d) 
and optimal $K$ (e)-(f) with
exponential type-1 and type-2 jobs.}
\label{fig:ATIR_expo}
\end{figure*}

\begin{figure*}[t!]
\begin{subfigure}{.46\textwidth}
  \centering
  \includegraphics[width=1\linewidth]{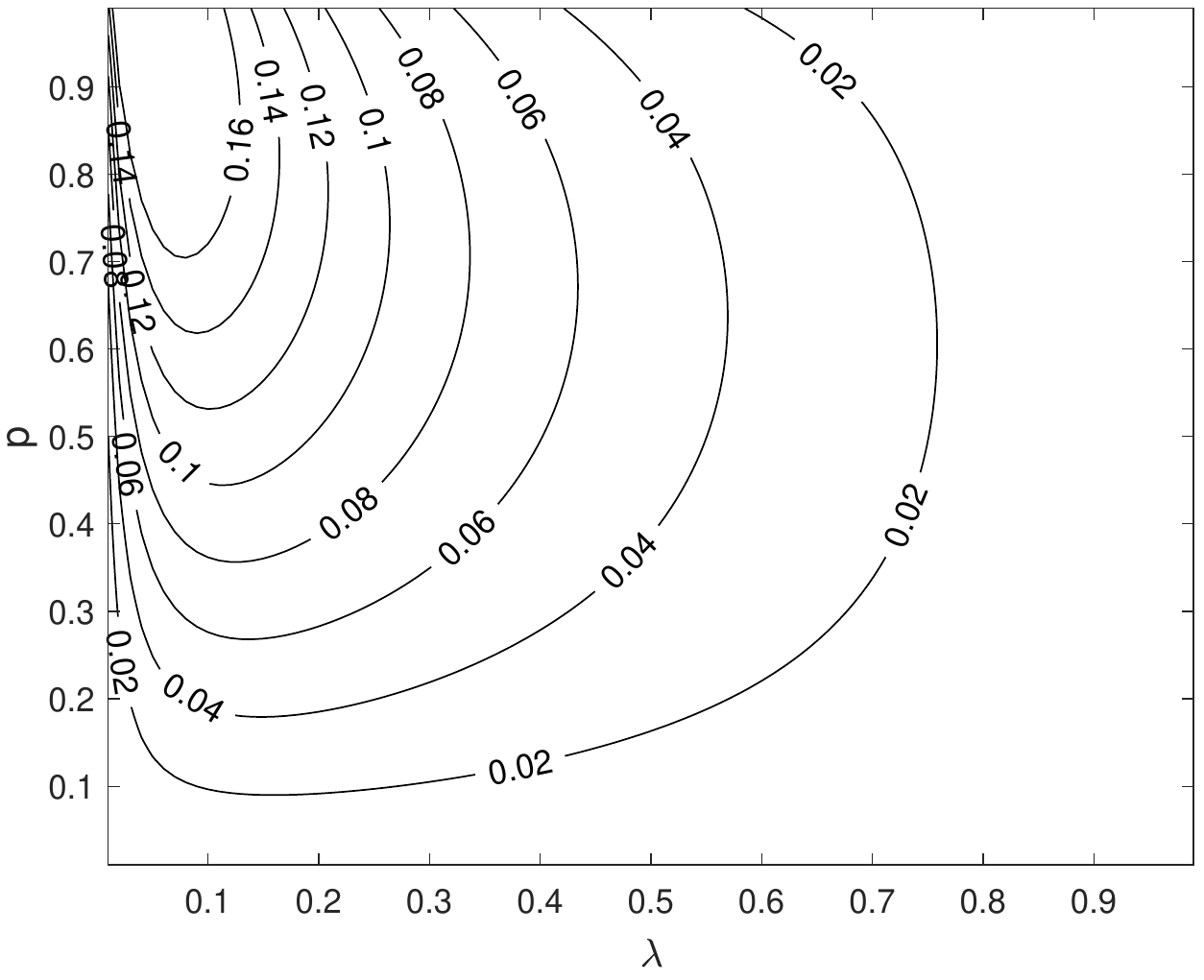}
	\caption{$K = 1$, $E[X_2]/E[X_1] = 5$}
	\label{fig:K1r5h}
\end{subfigure}
\begin{subfigure}{.46\textwidth}
  \centering
  \includegraphics[width=1\linewidth]{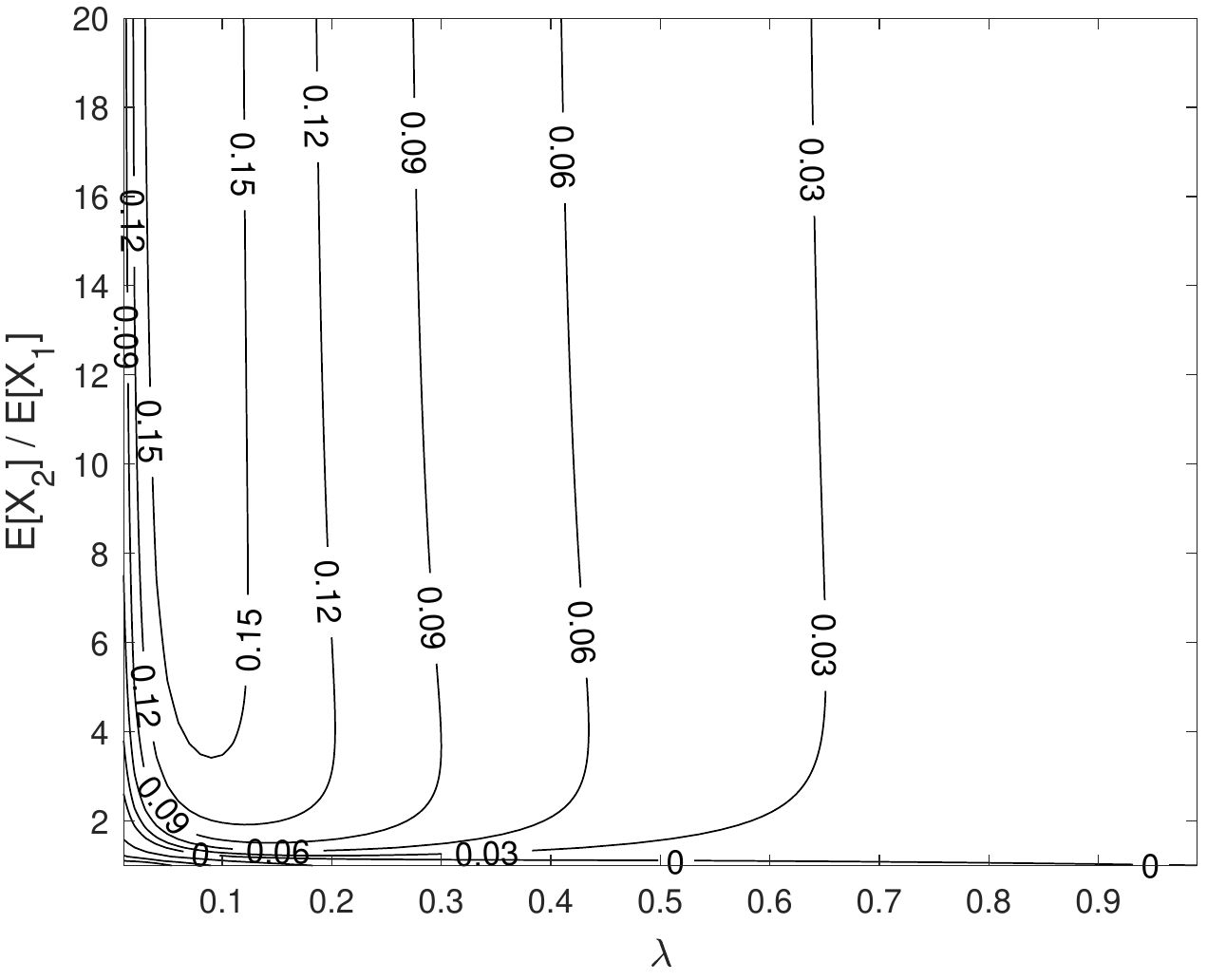}
	\caption{$K = 1$, $p = 0.7$}
	\label{fig:K1p07h}
\end{subfigure}
\begin{subfigure}{.46\textwidth}
  \centering
  \includegraphics[width=1\linewidth]{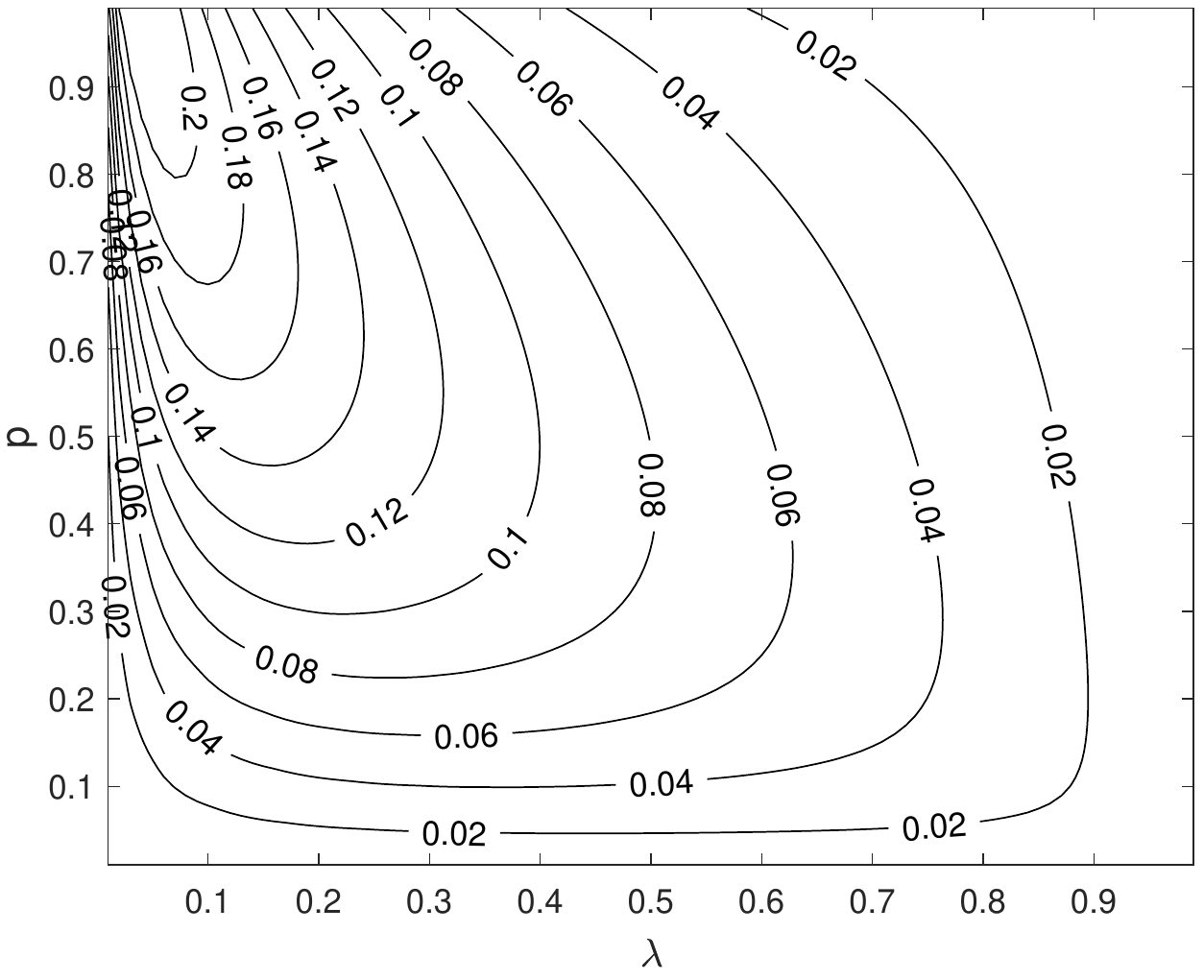}
	\caption{$K = K_{opt}$, $E[X_2]/E[X_1] = 5$}
	\label{fig:Kopt_r5h}
\end{subfigure}
\begin{subfigure}{.46\textwidth}
  \centering
  \includegraphics[width=1\linewidth]{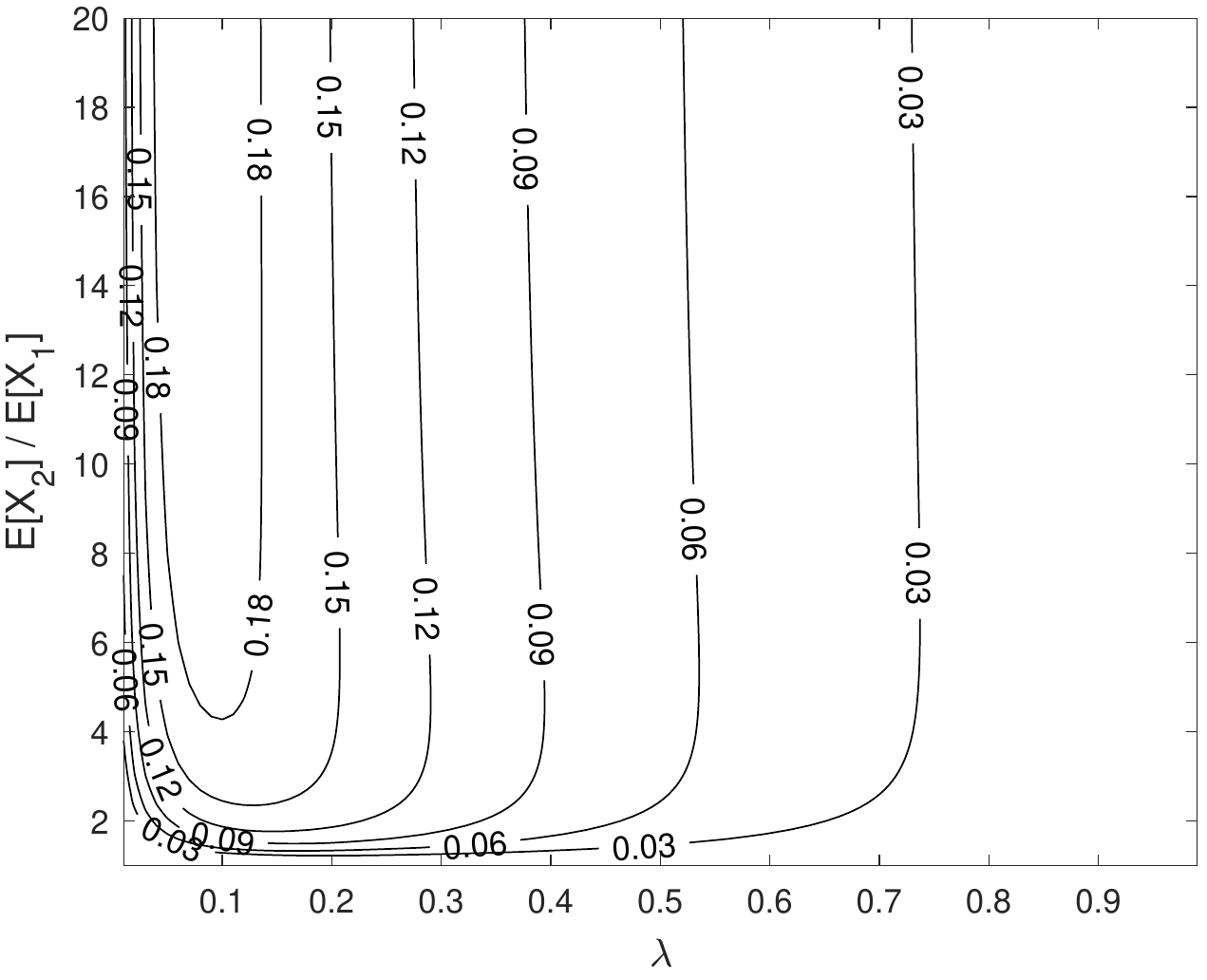}
	\caption{$K = K_{opt}$, $p = 0.7$}
	\label{fig:Kopt_p07h}
\end{subfigure}
\begin{subfigure}{.46\textwidth}
  \centering
  \includegraphics[width=1\linewidth]{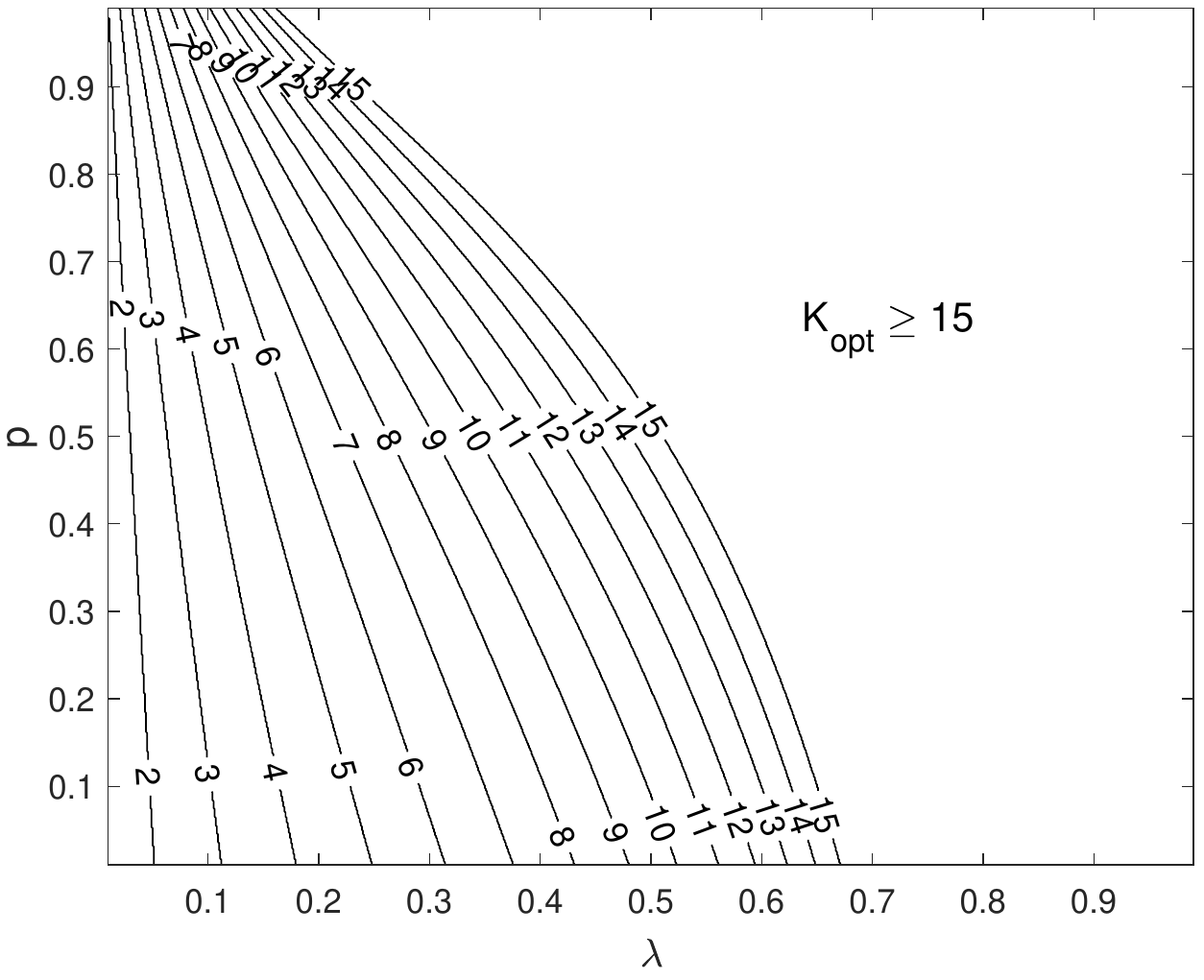}
	\caption{$E[X_2]/E[X_1] = 5$}
	\label{fig:H2best_r2}
\end{subfigure}
\begin{subfigure}{.46\textwidth}
  \centering
  \includegraphics[width=1\linewidth]{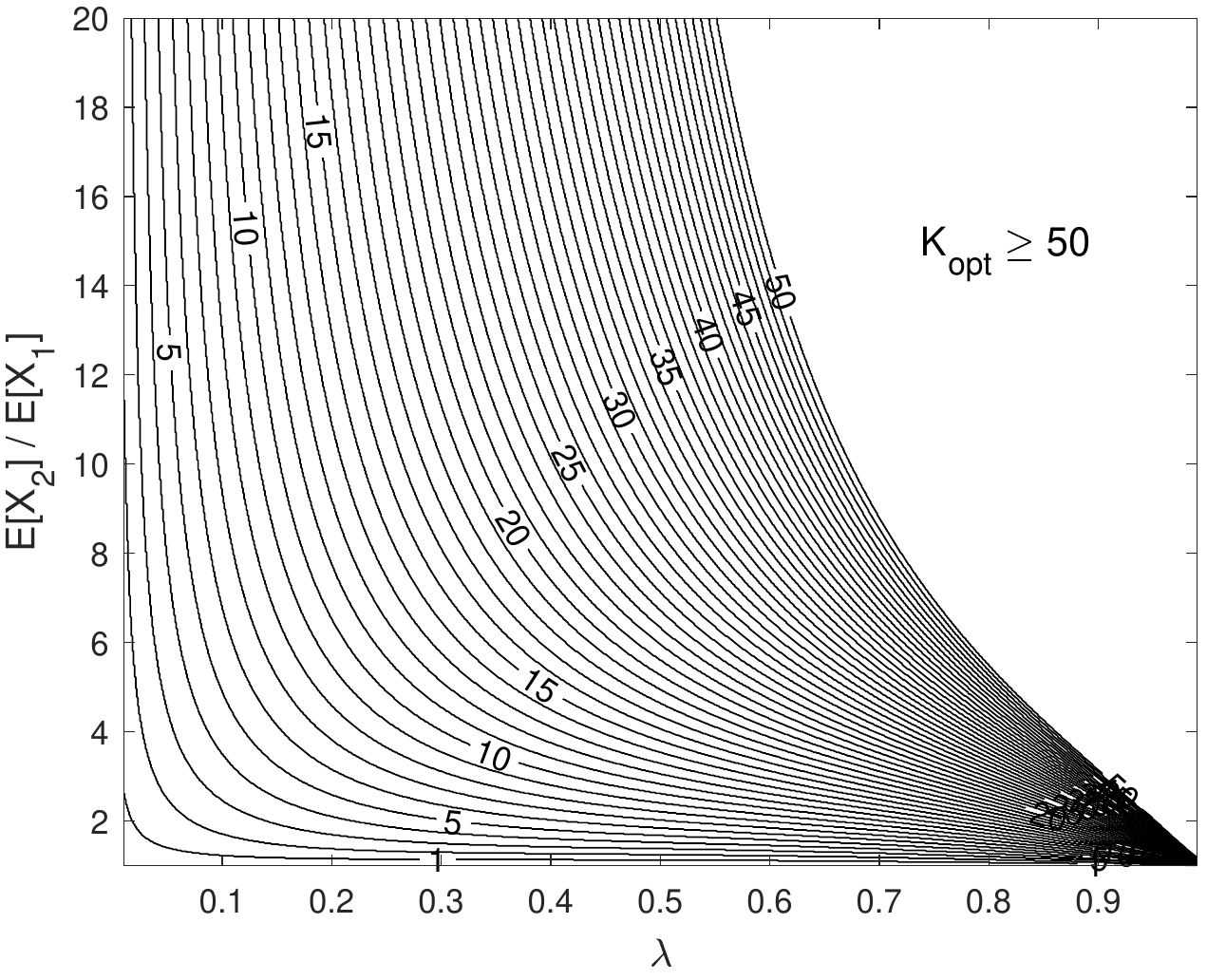}
	\caption{$p = 0.7$}
	\label{fig:H2best_p07}
\end{subfigure}

\caption{The asymptotic tail improvement ratio of Nudge-K over FCFS (a)-(d)
and optimal $K$ (e)-(f) with
hyper-exponential type-1 and type-2 jobs ($SCV = 5$ and balanced means).}
\label{fig:ATIR_h2}
\end{figure*}

\begin{figure*}[t!]
\begin{subfigure}{.46\textwidth}
  \centering
  \includegraphics[width=1\linewidth]{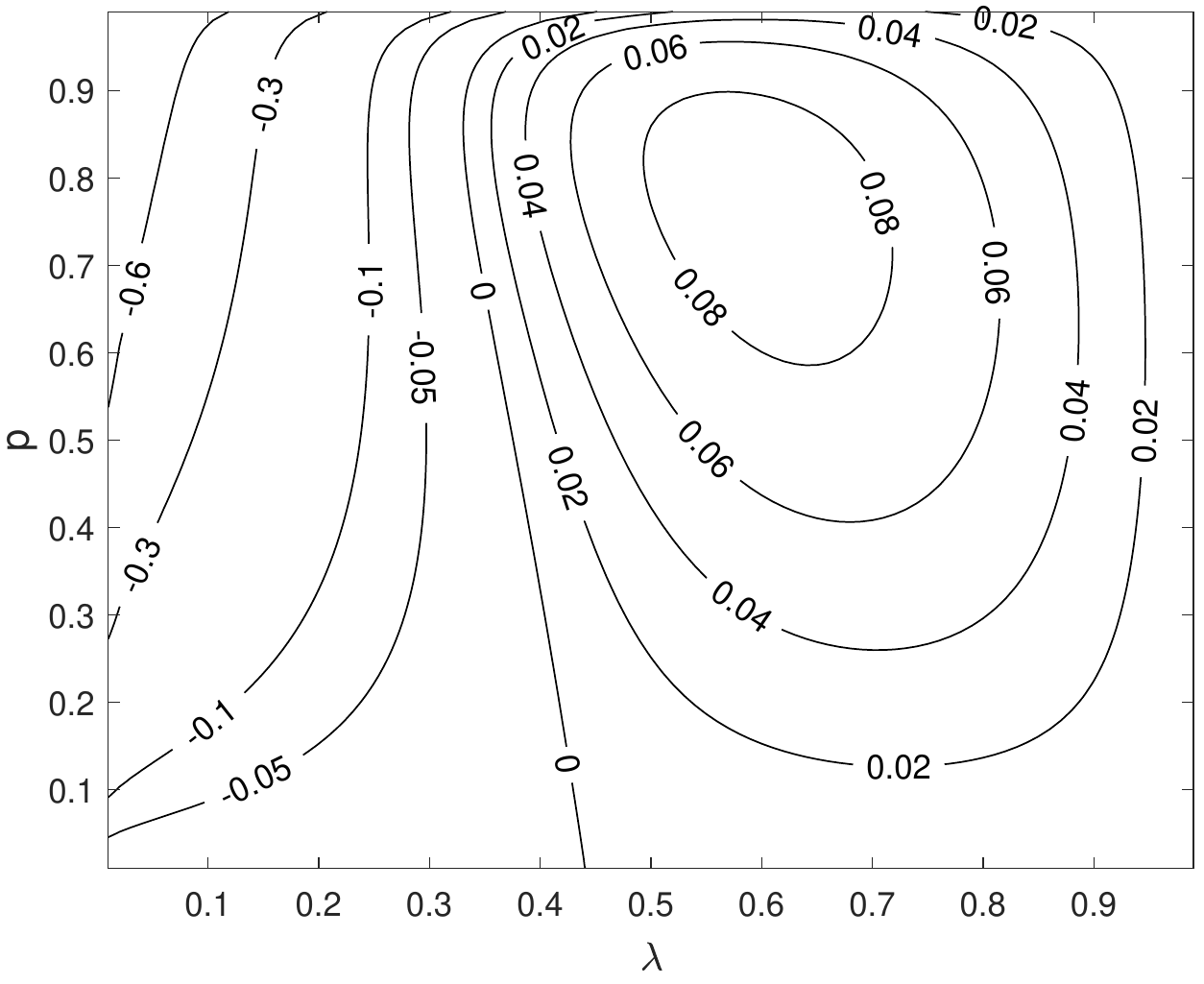}
	\caption{$K = 1$, $E[X_2]/E[X_1] = 5$}
	\label{fig:K1r5e}
\end{subfigure}
\begin{subfigure}{.46\textwidth}
  \centering
  \includegraphics[width=1\linewidth]{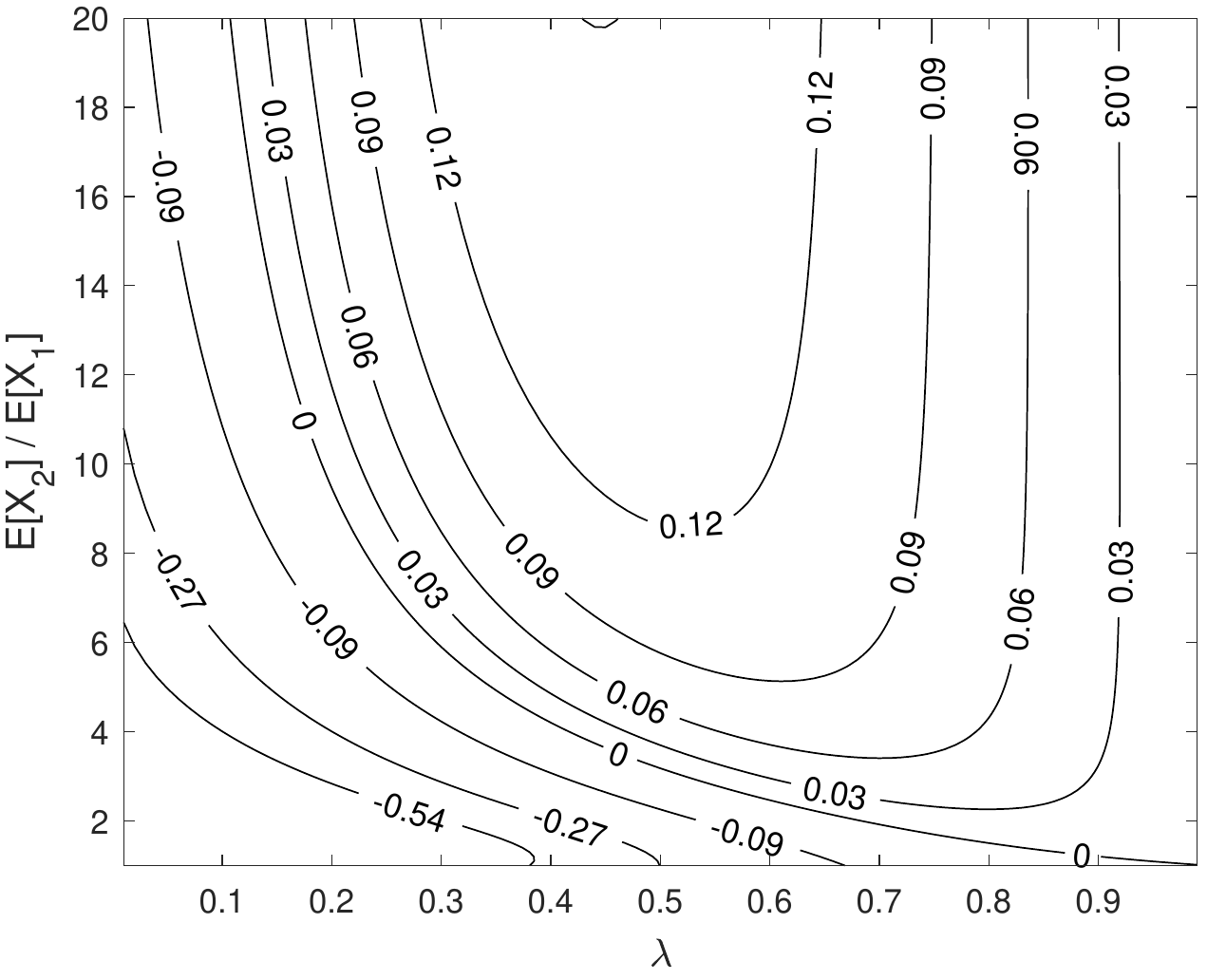}
	\caption{$K = 1$, $p = 0.7$}
	\label{fig:K1p07e}
\end{subfigure}
\begin{subfigure}{.46\textwidth}
  \centering
  \includegraphics[width=1\linewidth]{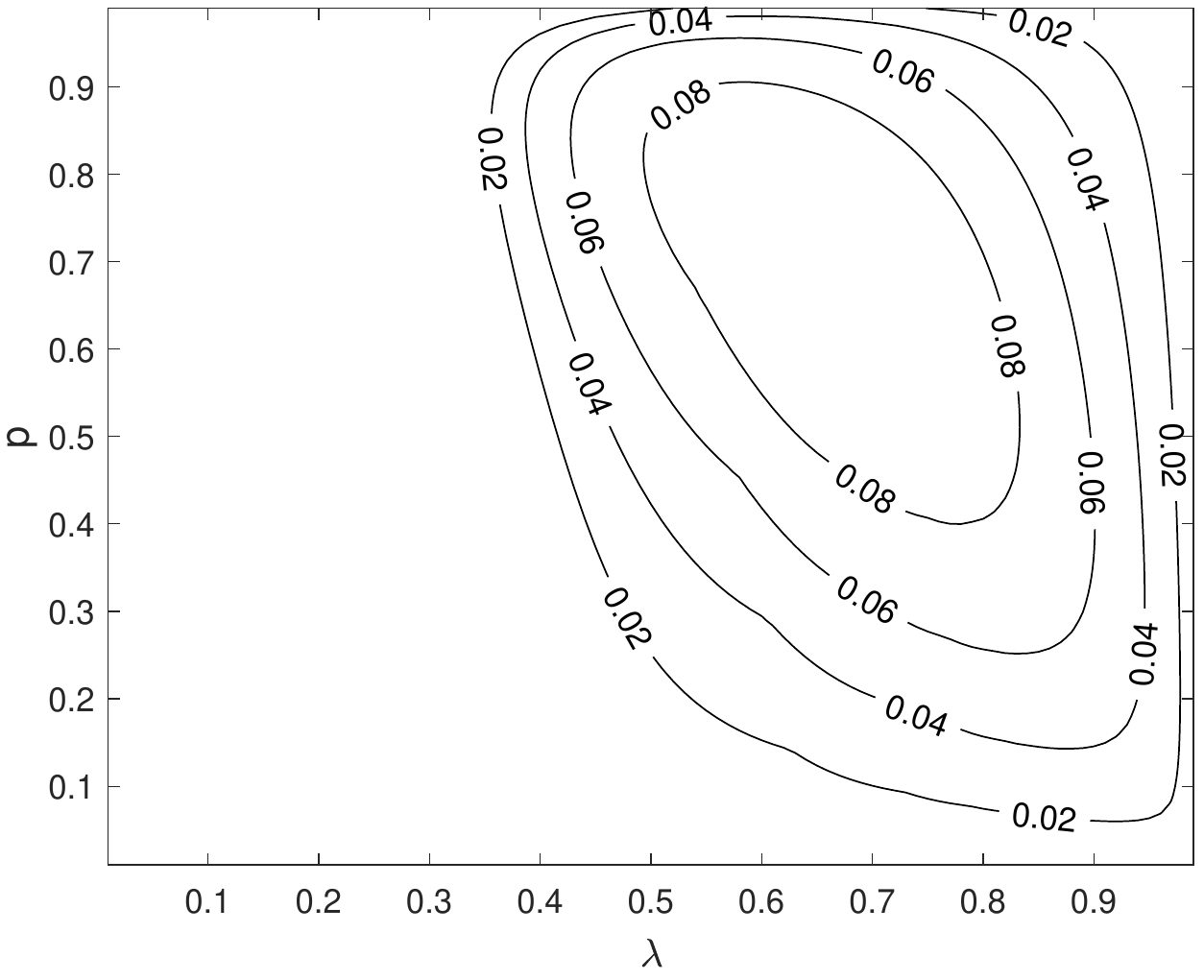}
	\caption{$K = K_{opt}$, $E[X_2]/E[X_1] = 5$}
	\label{fig:Koptr5e}
\end{subfigure}
\begin{subfigure}{.46\textwidth}
  \centering
  \includegraphics[width=1\linewidth]{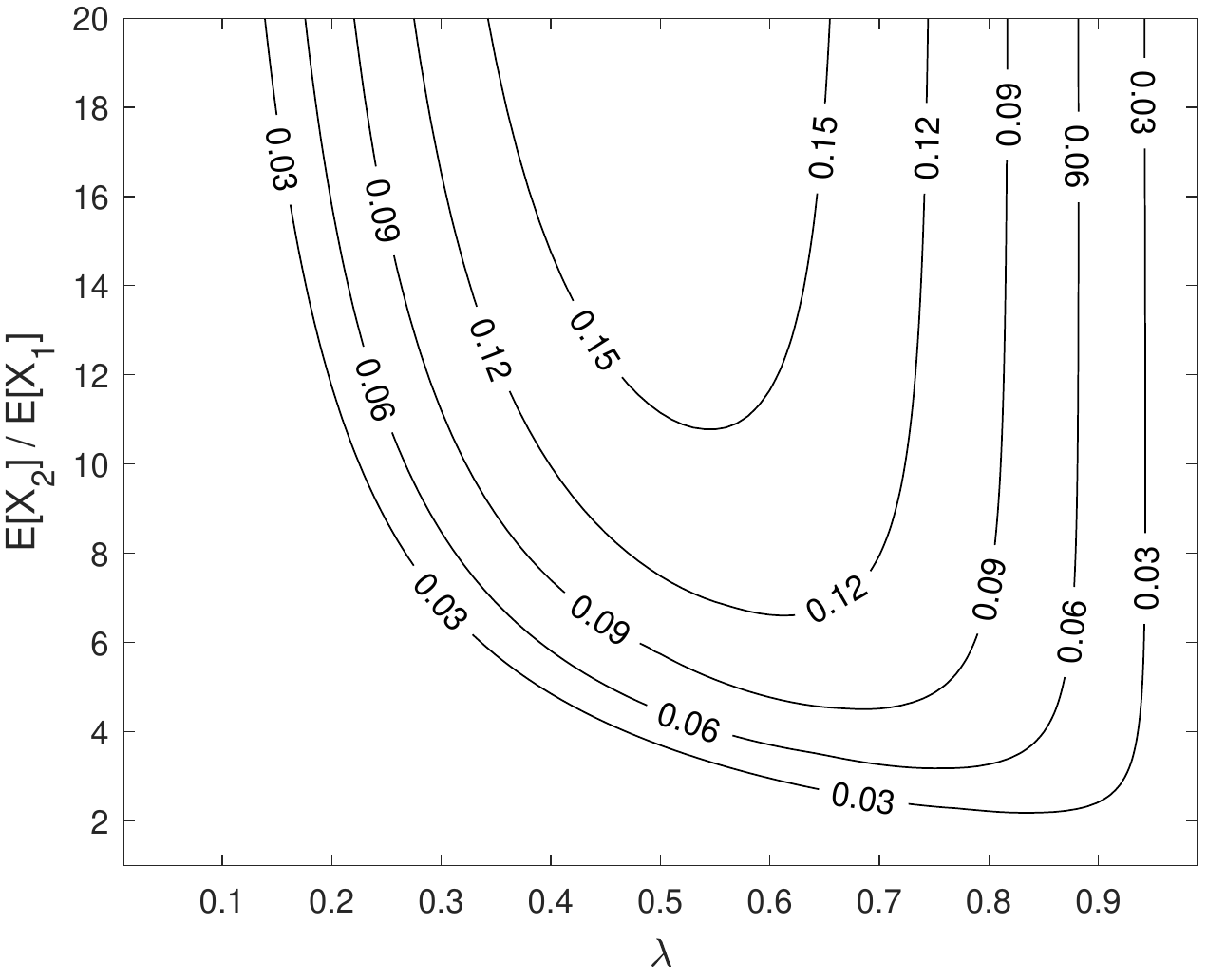}
	\caption{$K = K_{opt}$, $p = 0.7$}
	\label{fig:Koptp07e}
\end{subfigure}
\begin{subfigure}{.46\textwidth}
  \centering
  \includegraphics[width=1\linewidth]{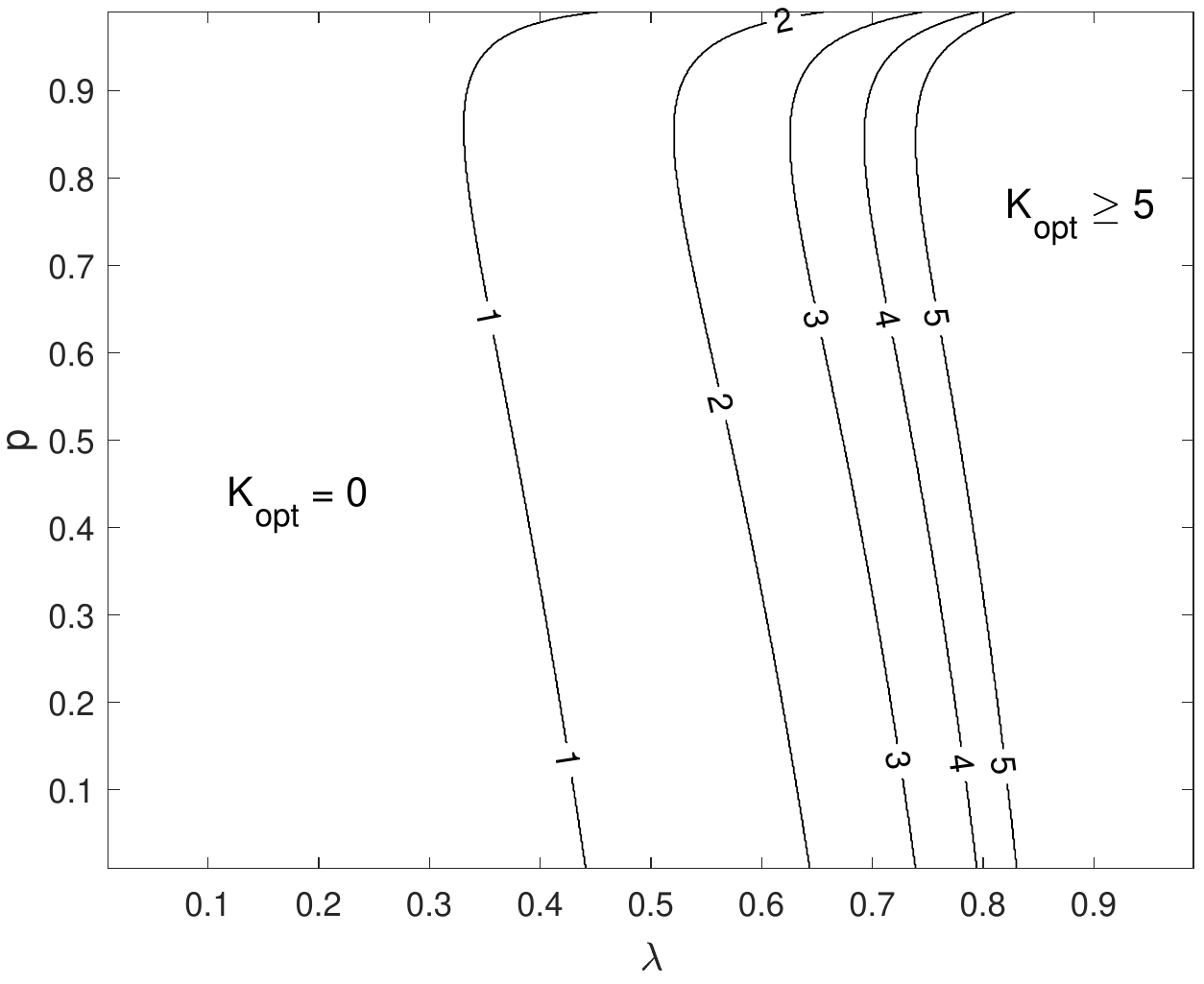}
	\caption{$E[X_2]/E[X_1] = 5$}
	\label{fig:Erlbest_r2}
\end{subfigure}
\hspace*{0.35cm}
\begin{subfigure}{.46\textwidth}
  \centering
  \includegraphics[width=1\linewidth]{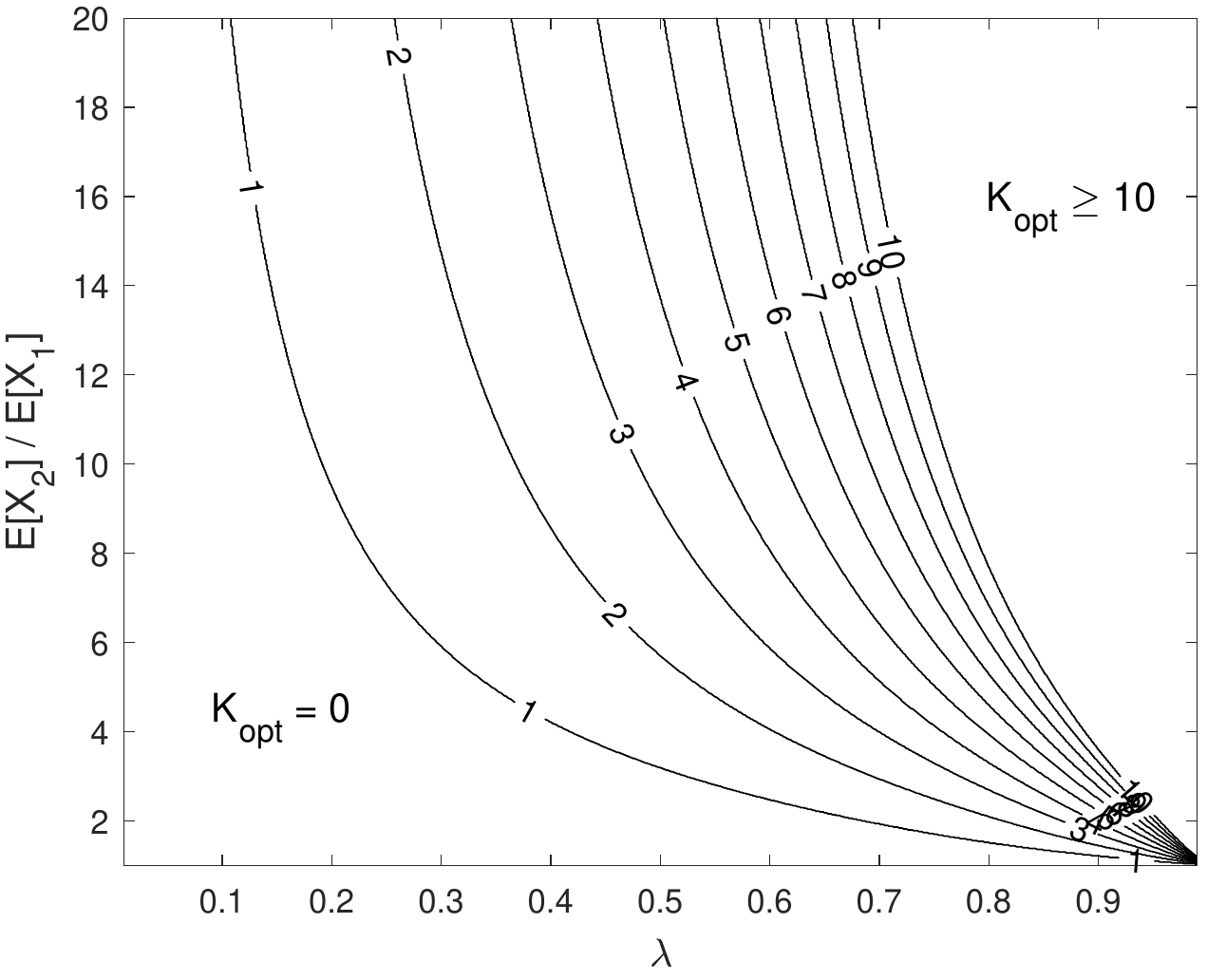}
	\caption{$p = 0.7$}
	\label{fig:Erlbest_p07}
\end{subfigure}
\caption{The asymptotic tail improvement ratio of Nudge-K over FCFS 
(a)-(d) and optimal $K$ (e)-(f) with
Erlang-$5$ type-1 and type-2 jobs.}
\label{fig:ATIR_erl}
\end{figure*}
\end{document}